\documentclass[12pt]{amsart}
\usepackage{amscd,amsmath,amsthm,amssymb}
\usepackage{color}
\usepackage{tikz}
\usepackage{url}
\usepackage{circuitikz}
\usepackage{float}

\newtheorem{Theorem}{Theorem}[section]
\newtheorem{Lemma}[Theorem]{Lemma}
\newtheorem{Corollary}[Theorem]{Corollary}
\newtheorem{Proposition}[Theorem]{Proposition}
\newtheorem{Remark}[Theorem]{Remark}

\newtheorem{Example}[Theorem]{Example}

\newtheorem{Definition}[Theorem]{Definition}

\newtheorem{Conj}[Theorem]{Conjecture}

\def\qed{\ifhmode\textqed\fi
	\ifmmode\ifinner\quad\qedsymbol\else\dispqed\fi\fi}
\def\textqed{\unskip\nobreak\penalty50
	\hskip2em\hbox{}\nobreak\hfill\qedsymbol
	\parfillskip=0pt \finalhyphendemerits=0}
\def\dispqed{\rlap{\qquad\qedsymbol}}

\def\height{\textup{height}}

\def\depth{\textup{depth\,}}

\def\im{\textup{im}}
\def\pd{\textup{proj\,dim}}

\def\reg{\textup{reg}}
\def\im{\textup{im}}

\def\lex{\textup{lex}}

\def\NI{\textup{NI}}

\def\dist{\textup{dist}}

\begin{document}
	
	\title{$\textbf{k}$-neighborhood ideals of graphs}
	\author{Somayeh Moradi, Leila Sharifan}

		\address{Somayeh Moradi, Department of Mathematics, Faculty of Science, Ilam University, P.O.Box 69315-516, Ilam, Iran}
		\email{so.moradi@ilam.ac.ir}
	
		\address{Leila Sharifan, Department of Mathematics and Computer Sciences, Hakim Sabzevari University, P.O. Box 397, Sabzevar, Iran}
	\email{l.sharifan@hsu.ac.ir}

	
	\begin{abstract}
In this paper, we introduce and investigate the $\textbf{k}$-neighborhood ideal of a graph, a natural generalization of the closed neighborhood ideal. Let $G$ be a simple graph on the vertex set $[n]$, and let $S=K[x_1,\dots,x_n]$ be the polynomial ring over a field $K$. For a vector $\textbf{k}=(k_1,\ldots,k_n)\in \mathbb{N}^n$ satisfying $1\leq k_i\leq \textrm{deg}_G(i)+1$ for all $i$, the $\textbf{k}$-neighborhood ideal of $G$ is defined as the squarefree monomial ideal $$\textrm{NI}_{\textbf{k}}(G)=\sum_{i=1}^n\, (\textbf{x}_W:\, W\subseteq N_G[i],\, |W|=k_i)$$ of $S$, where $\textbf{x}_W=\prod_{i\in W} x_i$. We study homological invariants and properties of $\textrm{NI}_{\textbf{k}}(G)$ focusing on its Castelnuovo-Mumford regularity, projective dimension and Cohen-Macaulayness. Special attention is devoted to the case where the vector ${\textbf{k}}$ is the degree-vector of the graph, i.e., $k_i=\textrm{deg}_G(i)$ for all vertices $i$, and to the case where $\textrm{NI}_{\textbf{k}}(G)$ coincides with the edge ideal of a graph. In these settings, we provide combinatorial characterizations and bounds for the regularity and projective dimension of $\textrm{NI}_{\textbf{k}}(G)$ for several classes of graphs, and further investigate the Cohen-Macaulay property of these ideals.	\end{abstract}

\maketitle
\vspace*{-1.8em}

\section{introduction}


Let $G$ be a finite simple graph with the vertex set $V(G)=[n]$, and let $S=K[x_1,\ldots,x_n]$ denote the polynomial ring in $n$ variables over a field $K$. The \emph{closed neighborhood ideal} of $G$, introduced in \cite{SM}, is the squarefree monomial ideal \[ \NI(G)=\big(\prod_{j\in N_G[i]}x_j : i\in [n]\big) \] in $S$, where $N_G[i]$ denotes the closed neighborhood of the vertex $i$ in $G$. 
This construction reveals a strong connection between algebraic properties of $\NI(G)$ and classical combinatorial invariants of $G$, including domination number, vertex cover number, and matching number.
Following their introduction, closed neighborhood ideals have been the subject of considerable study, with numerous algebraic and homological properties examined in the literature. We briefly summarize several key developments and results relevant to these ideals.
In \cite{MS}, the authors proved that for every chordal graph $G$,
$
\reg(S/\NI(G))=\tau(G),
$
where $\tau(G)$ denotes the vertex cover number of $G$. This equality was previously established for trees by Chakraborty \emph{et al.}, in \cite[Theorem 1.1]{CJRS}. The same work also showed that, for arbitrary graphs, the matching number $a(G)$ provides a lower bound for $\reg(S/\NI(G))$. Further combinatorial descriptions of homological invariants of closed neighborhood ideals like regularity, projective dimension and $v$-number can be found in \cite{CJRS,JHR,MS,OO,SM}. Joseph, Roy, and Singh \cite{J} studied minimal free resolutions of closed neighborhood ideals of trees.  Fröberg \cite{F} and Namiq \cite{Namiq} investigated Stanley-Reisner rings associated to neighborhood complexes having linear resolutions. The neighborhood complex of $G$ was first introduced by Lov$\acute{}$asz and played a central role in his celebrated proof of Kneser’s conjecture \cite{L}. This complex appears as the Alexander dual of the
Stanley-Reisner simplicial complex of $\NI(G^c)$, known as the dominance complex, see \cite{MW}.
The dominance complex of $G$, was studied by Matsushita and Wakatsuki \cite{MW} from a topological perspective. 
The Cohen-Macaulay property of closed neighborhood ideals has also been investigated in several works. Honeycutt and Sather-Wagstaff \cite{HS} characterized the Cohen-Macaulay closed neighborhood ideals of trees, whereas Leaman \cite{L} obtained analogous characterizations for chordal and bipartite graphs. Subsequently, the authors \cite{MS} established a characterization for very well-covered graphs. Furthermore, the componentwise linearity and regularity of squarefree powers of closed neighborhood ideals were examined by Nambi and Qureshi in \cite{NQ}. For results concerning associated primes of powers and normality properties of closed neighborhood ideals we refer the reader to \cite{HV1,N1,N2,N3}.
Several related classes of ideals have recently been introduced and investigated. These include open neighborhood ideals and $t$-closed neighborhood ideals, studied respectively by Lim \emph{et al.} \cite{Lim1,Lim} and Sharifan \cite{S1}.

Motivated by these developments, in this paper we introduce a broad generalization of the closed neighborhood ideal, called the \emph{$\mathbf{k}$-neighborhood ideal} of a graph. For a vector ${\bf k}=(k_1,\ldots,k_n)\in \mathbb{N}^n$ with $1\leq k_i\leq \deg_G(i)+1$, the $\mathbf{k}$-neighborhood ideal of $G$, denoted by $\NI_{\mathbf{k}}(G)$, is the ideal generated by the squarefree monomials $\prod_{j\in W} x_j$, where $W\subseteq N_G[i]$ with $|W|=k_i$ for some vertex $i$ of $G$.
This construction interpolates between several well-known classes of squarefree monomial ideals. When $k_i=\deg_G(i)+1$ for all $i$, one recovers the closed neighborhood ideal, while for complete graphs and the vector $\mathbf{k}$ with $k_i=d$ for all $i$,  the ideal recovers a squarefree Veronese ideal generated in degree $d$. The ${\bf k}$-neighborhood ideal of $G$ also appears as the edge ideal of another graph in certain cases.
The main objective of this work is to investigate the relationship between algebraic properties of $\NI_{\mathbf{k}}(G)$ and the structure and invariants of $G$.

The paper is organized as follows.
In Section 2, we introduce the $\bf k$-neighborhood ideal and establish its basic properties, including its relation to multiple domination in graphs. In Lemma~\ref{primaryDecom}, a correspondence between minimal prime ideals of $\NI_{\bf k}(G)$ and minimal multiple dominating sets of the graph is shown, which determines the height of $\NI_{\bf k}(G)$  in terms of generalized domination parameters. Motivated by this correspondence, we introduce the multiple dominating ideal associated to a graph. Using this construction, Proposition~\ref{ComBipart} determines the regularity and projective dimension of $\NI_{\mathbf{k}}(G)$ for complete bipartite graphs. In Proposition~\ref{m-vertex}, a combinatorial lower bound for the projective dimension of $\NI_{\mathbf{k}}(G)$ is given. Section~3 focuses on the case where $\mathbf{k}$ is the degree-vector of $G$, namely $k_i=\deg_G(i)$ for all vertices $i$. While \cite[Theorem 3.2]{MS} shows that $\reg(S/\NI(G))=\tau(G)$ for chordal graphs, Proposition~\ref{chordalReg} demonstrates that, in contrast, the difference $\tau(G)-\reg(S/\NI_{\mathbf{k}}(G))$ can be arbitrarily large when $\mathbf{k}$ is the degree-vector. We therefore conjecture that \[ \reg(S/\NI_{\mathbf{k}}(G))\leq \tau(G)-c(G) \] for every chordal graph $G$ with no isolated vertices, where $c(G)$ denotes the number of connected components of $G$ (Conjecture~\ref{chordalbound}). We verify this conjecture for several families of chordal graphs. Lemma~\ref{path} provides an explicit formula for $\reg(S/\NI_{\mathbf{k}}(P_n))$, where $P_n$ is a path graph. Using this result, Theorem~\ref{caterpillar} establishes the conjecture for caterpillar graphs, while Theorem~\ref{clique-star} proves it for clique-star graphs and generalized caterpillar graphs. 
 Theorem~\ref{chordalCMgeneralized} presents a family of graphs for which $S/\NI_{\mathbf{k}}(G)$ is Cohen-Macaulay, and $\reg(S/\NI_{\bf k}(G))$ can be described combinatorially.  The section concludes with Corollary~\ref{product1} and Corollary~\ref{product2}, which give lower bounds for $\pd(S/\NI_{\bf k}(G))$ for some graphs, using the concept of an $\mathbf{m}$-cover introduced in Section 2.
In Section~4, we investigate properties of $\NI_{\mathbf{k}}(G)$ in the cases that $\NI_{\mathbf{k}}(G)$ is the edge ideal of a graph. A fundamental example occurs when $k_i=2$ for all $i$, in which case \[ \NI_{\mathbf{k}}(G)=I(\widetilde{G}), \] where $\widetilde{G}$ denotes the second power of $G$. Additional instances are described in Propositions~\ref{dominatingset} and \ref{edgeIdeal2}. Our main emphasis is on the case $k_i=2$ for all $i$. By interpreting combinatorial invariants and structural properties of $\widetilde{G}$ in terms of the original graph $G$, we derive precise descriptions and bounds for $\reg(S/\NI_{\mathbf{k}}(G))$ for several graph classes, including claw-free chordal graphs, block graphs, graphs with  no cycle of length bigger than $5$, and striped maximal outerplane graphs (see Proposition~\ref{claw-free}, Proposition~\ref{block}, Proposition~\ref{small cycle}, Proposition~\ref{outerplane}, and Corollary~\ref{regEdgeIdealCase}).  Theorem~\ref{CMtree} provides a combinatorial characterization of the Cohen-Macaulay property of $S/\NI_{\mathbf{k}}(T)$ when $T$ is a tree. Finally, in Corollaries~\ref{product3} and  ~\ref{product4} lower bounds for the projective dimension of $\NI_{\mathbf{k}}(G)$ is given for products of path graphs.

	\section{Basic properties of $\bf{k}$-neighborhood ideals}

In this section, we present the minimal primary decomposition of a $\mathbf{k}$-neighborhood ideal  and apply this result to determine homological invariants of $\NI_{\mathbf{k}}(G)$, when $G$ is a complete bipartite graph. Moreover we give a lower  bound  for the projective dimension of the $\mathbf{k}$-neighborhood ideal of an arbitrary graph $G$.

Throughout this paper, $G$ denotes a finite simple graph without isolated vertices on the vertex set $[n]$, and ${\mathbf k}=(k_1,\ldots,k_n)\in \mathbb{N}^n$ is a vector satisfying $1\leq k_i\leq \deg_G(i)+1$ for all $i\in [n]$. Here, $\deg_G(i)$ denotes the degree of $i$ in $G$. When the graph is clear from the context, we simply write $\deg(i)$ instead of $\deg_G(i)$. The vertex set and edge set of $G$ are denoted by $V(G)$ and $E(G)$, respectively. For a subset $A\subseteq [n]$, we set $\textbf{x}_A=\prod_{i\in A} x_i$, and for any vertex $i\in V(G)$, we set $u_{i,G}=\textbf{x}_{N_G[i]}$ to be a monomial in $S=K[x_1,\ldots,x_n]$.

	\smallskip
	
	 The {\em $\bf{k}$-neighborhood ideal} of $G$ is defined as
	$$\NI_{\bf k}(G)=\sum_{i=1}^n\, (\textbf{x}_W:\, W\subseteq N_G[i],\, |W|=k_i).$$
	
	Notice that if $k_i=\deg(i)+1$ for all $i\in [n]$, then the $\bf{k}$-neighborhood ideal of $G$ is just the closed neighborhood ideal $\NI(G)$. 
	
	We start with an example which shows that a $\bf{k}$-neighborhood ideal  may appear as a squarefree Veronese ideal or as an edge ideal. A squarefree Veronese ideal on $n$ variables and generated in degree $d$ is denoted by $I_{n,d}$.

	\begin{Example} 
		Let $G$ be a graph on $[n]$. Then
		\begin{itemize}
			\item[(1)] If $G=K_n$ and ${\bf k}=(d,\ldots,d)\in \mathbb{N}^n$ for some $1\leq d\leq n$, then $\NI_{\bf k}(G)=I_{n,d}$.
			\item[(2)] Let $G$ be a bipartite graph and ${\bf k}=(2,\ldots,2)\in \mathbb{N}^n$. Then  $\NI_{\bf k}(G)=I(K_n)$ if and only if $G$ is a complete bipartite graph.
			\item[(3)] Let $G$ be an $(n-2)$-regular graph and ${\bf k}=(d,\ldots,d)\in \mathbb{N}^n$, where $1\leq d\leq n-1$ is an integer. The complementary graph $G^c$ is the disjoint union of edges.  Hence, for any $i\in [n]$ there exists $j\in [n]$ such that $N_G[j]=[n]\setminus \{i\}$. This implies that
		 $\NI_{\bf k}(G)=I_{n,d}$.
			\item[(4)] Let $G=K_{2,n-2}$ be a complete bipartite graph with a bipartite partition $\{1,2\}\sqcup\{3,\ldots,n\}$, and let ${\bf k}=(n-2,n-2,2,\ldots,2)\in \mathbb{N}^n$. Then $$\NI_{\bf k}(G)=I(K_{2,n-2})+(x_1x_2\,,\,x_3x_4\cdots x_n).$$
		\end{itemize}
	\end{Example}

	In \cite[Lemma 2.2]{SM}, it was shown that the minimal prime ideals of $\NI(G)$ correspond to the minimal dominating sets of $G$. In order to describe the minimal prime ideals of $\NI_{\mathbf{k}}(G)$, we first introduce an appropriate generalization of dominating sets.
	
	There are several extensions of the classical notion of a dominating set in graph theory. One such generalization is that of an \emph{$h$-tuple dominating set}, and more generally, the notion of an \emph{$\mathbf{m}$-multiple dominating set}, where $\mathbf{m}=(m_1,\ldots,m_n)\in\mathbb{N}^n$. These concepts were first introduced in~\cite{HH}. Let $\mathbf{m}=(m_1,\ldots,m_n)\in\mathbb{N}^n$ satisfy $1 \le m_i \le \deg(i)+1$ for all $i\in[n]$. A subset $S \subseteq [n]$ is called an \emph{$\mathbf{m}$-multiple dominating set} of $G$ if
	\[
	|N_G[i] \cap S| \ge m_i \quad \text{for all } i \in [n].
	\]
	In particular, $S$ is called an \emph{$h$-tuple dominating set} if $m_i = h$ for all $i \in [n]$.
	We denote by $\gamma(G,\mathbf{m})$ the minimum cardinality of an $\mathbf{m}$-multiple dominating set of $G$. This notion was first studied in~\cite{HH}.
	
	\begin{Lemma}\label{primaryDecom}
		
		Let \(G\) be a graph on the vertex set \([n]\), and let  ${\bf k}=(k_1,\ldots,k_n)\in \mathbb{N}^n$ satisfy  $1\leq k_i\leq \deg(i)+1$ for all $i$. Define \(\mathbf{m}=(m_1,\ldots,m_n)\) by
		\[
		m_i = \deg(i) - k_i + 2 \quad \text{for each } i\in[n].
		\]
		Then a monomial prime ideal \(P=(x_{j_1},\ldots,x_{j_\ell})\) is a minimal prime ideal of \(\mathrm{NI}_{\mathbf{k}}(G)\) if and only if the set \(\{j_1,\ldots,j_\ell\}\) is a minimal \(\mathbf{m}\)-multiple dominating set of \(G\).
	\end{Lemma}
	\begin{proof}
	It suffices to show that  $\NI_{\bf k}(G)\subseteq (x_{j_1},\ldots, x_{j_\ell}) $ if and only if  $\{j_1,\ldots, j_l\}$ is an ${\bf{m}}$-multiple dominating set of $G$. Suppose that $S=\{j_1,\ldots,j_\ell\}$ is an $\mathbf{m}$-multiple dominating set of $G$, and let $\mathbf{x}_W \in \mathrm{NI}_{\mathbf{k}}(G)$, where $W \subseteq N_G[i]$ with $|W|=k_i$ for some $i \in [n]$. Since $S$ is an $\mathbf{m}$-multiple dominating set, we have $|S \cap N_G[i]| \ge m_i$. Hence, we may decompose
	$
	N_G[i] = A \sqcup B,
	$
	where $A \subseteq S$ with $|A| = m_i$ and $B = N_G[i] \setminus A$. Then
	\[
	|B| = \deg(i)+1 - m_i < k_i = |W|.
	\]
	It follows that $W \nsubseteq B$, and consequently $W \cap A \neq \emptyset$. Since $A \subseteq S$, we conclude that $W \cap S \neq \emptyset$, which implies $\mathbf{x}_W \in (x_{j_1},\ldots,x_{j_\ell})$.
	
To prove the ``only if'' part, suppose that 
$\mathrm{NI}_{\mathbf{k}}(G)\subseteq (x_{j_1},\ldots,x_{j_\ell}),$
	and by contradiction assume that
	$|N_G[i]\cap \{j_1,\ldots,j_\ell\}| < m_i$
	for some \(i \in [n]\). Then it follows that
	\[
	|N_G[i]\setminus \{j_1,\ldots,j_\ell\}| \ge \deg(i)+2 - m_i.
	\]
	Hence, we may choose a subset \(W \subseteq N_G[i]\setminus \{j_1,\ldots,j_\ell\}\) with
	$|W| = \deg(i)+2 - m_i=k_i.$
	Consequently, \(\mathbf{x}_W \in \mathrm{NI}_{\mathbf{k}}(G)\), while \(\mathbf{x}_W \notin (x_{j_1},\ldots,x_{j_\ell})\), yielding a contradiction.
	\end{proof}

	\begin{Corollary}\label{height}
	We have $\height(\NI_{\bf k}(G))=\gamma(G,{\bf m})$,
		where ${\bf m}=(m_1,\ldots,m_n)$ is the vector defined by $m_i=\deg(i)-k_i+2$ for all $i\in[n]$.  
	\end{Corollary}

	For a vector \(\mathbf{m}=(m_1,\ldots,m_n)\in \mathbb{N}^n\) satisfying \(1 \le m_i \le \deg(i)+1\) for all \(i \in [n]\), we define the \emph{\(\mathbf{m}\)-multiple dominating ideal} of \(G\) to be the monomial ideal
	\[
	DI_{\mathbf{m}}(G) = \big(\mathbf{x}_D \;:\; D \text{ is a minimal } \mathbf{m}\text{-multiple dominating set of } G\big).
	\]
	
	By Lemma \ref{primaryDecom}, it follows that
	$
	\mathrm{NI}_{\mathbf{k}}(G)^\vee = DI_{\mathbf{m}}(G),
	$
	where \(m_i = \deg(i) - k_i + 2\) for all \(i \in [n]\).
	
	We will make use of this duality to investigate the regularity and projective dimension of \(\mathrm{NI}_{\mathbf{k}}(G)\) in the case where \(G\) is a complete bipartite graph.
	
	Recall that a monomial ideal $I$ has {\em linear quotients} if there exists an order $u_1,\ldots,u_m$ on the set of minimal monomial generators of $I$ such that the ideal $(u_1,\ldots,u_i):(u_{i+1})$ is generated by variables for any $1\leq i\leq m-1$.
	
	\begin{Proposition}\label{ComBipart} Let \(G=K_{r,s}\) be the complete bipartite graph with bipartition $ V(G)=X\sqcup Y$, where $X=\{x_1,\ldots,x_r\}$ and $Y=\{y_1,\ldots,y_s\}$. Let $m \leq r \leq s$,   \(\mathbf{m}=(m,\ldots,m)\in \mathbb{N}^n\)  and $\mathbf{k}=(k_1,\ldots,k_n)$ with $k_i=\deg(i)-m+2$ for all $i$. The following statements hold: \begin{enumerate} 
			\item[(a)] If $m<r$, then  \(DI_{\mathbf m}(G)\) is minimally generated by 
				\[
			\begin{aligned}
				&\{ \mathbf{x}_A \mathbf{y}_B : A \subset [r],\; B \subset [s],\; |A| = |B| = m \}
				\cup \{ \mathbf{x}_A \mathbf{y}_{[s]} : A \subset [r],\; |A| = m-1 \} \\
				&\qquad \cup \{ \mathbf{x}_{[r]} \mathbf{y}_B : B \subset [s],\; |B| = m-1 \},
			\end{aligned}
			\] 
			 where \(\mathbf{x}_A=\prod_{i\in A}x_i\) and \(\mathbf{y}_B=\prod_{j\in B}y_j\). 
			
			\item[(b)] If $m=r$, then \(DI_{\mathbf m}(G)\) is minimally generated by \[ \bigl\{\mathbf{x}_A\mathbf{y}_{[s]} : A\subset [r],\; |A|=m-1\bigr\} \;\cup\; \bigl\{\mathbf{x}_{[r]}\mathbf{y}_B : B\subset [s],\; |B|=m-1\bigr\}. \] 
			
			\item[(c)] \(DI_{\mathbf m}(G)\) has linear quotients. 
			
			\item[(d)] $\pd(S/\NI_{\bf k}(G))=\reg(DI_{\bf m}(G))=s+m-1$.
			
			\item[(e)] $\reg(S/\NI_{\bf k}(G))=\pd(DI_{\bf m}(G))= \begin{cases} r+s-2m, & \text{if } m<r,\\[4pt] s-m+1, & \text{if }  m=r\ge 2,\\
			1, & \text{if } m=r=1. \end{cases} $ \end{enumerate} 
		\end{Proposition}
		
	\begin{proof}

		(a), (b)
		A set \(D \subseteq V(G)\) is an \(\mathbf{m}\)-multiple dominating set of \(G\) if and only if one of the following conditions holds:
	 $|D\cap X| \ge m $ and $|D\cap Y| \ge m$,
		or
		$|D\cap X| = m-1 \ \text{and} \ Y \subseteq D,$
		or
		$|D\cap Y| = m-1 \ \text{and} \ X \subseteq D$.
		Consequently, the statements follow.
	\smallskip
		
		(c) Suppose $m<r$. Let $>_{\lex}$ denote the lexicographical order on $S$ induced by $x_1>\cdots>x_r>y_1>\cdots>y_s$. Let $u_1>_{\lex}  \cdots >_{\lex} u_p$ be all the monomials of degree $m$ in the variables $x_1,\ldots, x_r$, and let $v_1>_{\lex}  \cdots >_{\lex} v_q$ be all the monomials of degree $m-1$ in the same variables. Similarly, let $\tilde{u}_1>_{\lex} \cdots >_{\lex}\tilde{u}_{p'}$ be all the monomials of degree $m$ in the variables $y_1,\ldots, y_s$, and  $\tilde{v}_1>_{\lex}\cdots >_{\lex}\tilde{v}_{q'}$ be all the monomials of degree $m-1$ in these variables.
		 Then one can see that
		$$
		\begin{aligned}
			f_1 &= u_1\tilde{u}_1,\; \ldots,\; f_p = u_p\tilde{u}_1, \\
			f_{p+1} &= u_1\tilde{u}_2,\; \ldots,\; f_{2p} = u_p\tilde{u}_2,\\
			&\;\;\vdots\\
			f_{p(p'-1)+1} &= u_1\tilde{u}_{p'},\; \ldots,\; f_{pp'} = u_p\tilde{u}_{p'}, \\
			f_{pp'+1} &= v_1y_{[s]},\; \ldots,\; f_{pp'+q} = v_qy_{[s]}, \\
			f_{pp'+q+1} &= x_{[r]}\tilde{v}_1,\; \ldots,\; f_{pp'+q+q'} = x_{[r]}\tilde{v}_{q'}.
		\end{aligned}
		$$
	 forms an order of  linear quotients on the minimal monomial generators of $DI_{\mathbf{m}}(G)$. The case $m=r$ can be treated analogously.

			\smallskip
			
		(d) Since $DI_{\mathbf{m}}(G)$ has linear quotients, by \cite[Corollary 2.7]{SV}, $\reg(DI_{\mathbf{m}}(G))$ is equal to the maximum degree of elements in the minimal monomial generating set of $DI_{\mathbf{m}}(G)$. By (a) and (b) this degree is equal to $s+m-1$.   
	The equality $\pd(S/\NI_{\bf k}(G))=\reg(DI_{\bf m}(G))$ follows from \cite[Theorem 2.1]{T} and the fact that $	\mathrm{NI}_{\mathbf{k}}(G)^\vee = DI_{\mathbf{m}}(G)$. 
		
			\smallskip
			
		(e) Suppose that $m<r$. For each $1\leq \ell\leq pp'+q+q'$ let $n_\ell$ be the number of elements in the minimal monomial system of generators of $(f_1,\ldots,f_{\ell-1}):(f_\ell)$. Then by \cite[Corollary 2.7]{SV}, $$\pd(DI_{\mathbf{m}}(G))=\max\{n_\ell : 1\leq \ell\leq pp'+q+q'\}.$$ One can easily see that $\max\{n_\ell : 1\leq \ell\leq pp'+q+q'\}=n_{pp'}$ and $(f_1,\ldots,f_{pp'-1}):(f_{pp'})=(x_1,\dots,x_{r-m},y_1,\ldots,y_{s-m})$. Hence, $n_{pp'}=r+s-2m$. This together with \cite[Theorem 2.1]{T} implies that $\reg(S/\NI_{\bf k}(G))=\pd(DI_{\bf m}(G))=r+s-2m$.  
		
		If $m=r\ge2$, using the analogous order on the generators in (c), the same linear quotients argument shows that the maximum number of variables occurring in a colon ideal is $s-m+1$. Hence, $\pd(DI_{\mathbf{m}}(G))=s-m+1$.
	  Now, let $m=r=1$. Then $DI_{\mathbf{m}}(G)=(x_1,y_1\cdots y_s)$ which is a complete intersection, and hence $\pd(DI_{\mathbf{m}}(G))=1$.
	\end{proof}

		\begin{Definition}
		Let $G$ be a graph on $[n]$, and let ${\bf m}=(m_1,\ldots,m_n)\in \mathbb{Z}^n$ be a vector such that $0\leq m_i\leq \deg(i)$. We say that  $C \subseteq [n]$ is an ${\bf m}$-vertex cover of $G$ if 
		$$
		\begin{cases}
			|N_G[i]\cap C|\le m_i, & \forall\, i\in C,\\[4pt]
			|N_G[i]\cap C|= m_i, & \forall\, i\notin C.
		\end{cases}
		$$	
	\end{Definition}
	
	Note that an $\mathbf{m}$-vertex cover of a graph $G$ does not necessarily exist for every vector
	$\mathbf{m} = (m_1,\ldots,m_n) \in \mathbb{Z}^n$ with  $0\leq m_i\leq \deg(i)$. For instance, if $G = P_3$ and
	$\mathbf{m} = (0,2,1)$, then $G$ does not admit an $\mathbf{m}$-vertex cover.

	On the other hand, any minimal vertex cover of a graph $G$ is an
	$\mathbf{m}$-vertex cover of $G$ for $\mathbf{m} = (\deg(1),\ldots,\deg(n))$.
	Moreover, if $C = [n]\setminus D$, where $D$ is an arbitrary dominating set of $G$, then
	$|N_G[i]\cap C| \le \deg(i)$ for every $i \in [n]$. Hence, $C$ is an
	$\mathbf{m}$-vertex cover for each vector $\mathbf{m} = (m_1,\ldots,m_n)$ satisfying
	$0 \le m_i \le \deg(i)$ and
	\[
	\begin{cases}
		|N_G[i]\cap C| \le m_i, & \forall\, i \in C,\\[4pt]
		|N_G[i]\cap C| = m_i, & \forall\, i \notin C.
	\end{cases}
	\]

	In the following we present a lower  bound  for the projective dimension of $\NI_{\bf k}(G)$ in terms 
	of ${\bf m}$-vertex covers of $G$.  
	 
	\begin{Proposition}\label{m-vertex}
		Let $G$ be a graph on $[n]$, and let ${\bf k}=(k_1,\ldots,k_n)\in \mathbb{N}^n$ be a vector such that $1\leq k_i\leq \deg(i)+1$. Let $C$ be an ${\bf m}$-vertex cover of $G$, where $m_i=k_i-1$ for all $i$. Then 
		\begin{enumerate}
			\item [(a)]  $\pd(S/\NI_{\bf k}(G))\geq n-|C|$.
			\item [(b)] $\depth(S/\NI_{\bf k}(G))\leq |C|$.
		\end{enumerate}
	\end{Proposition}
	
	


	
	\begin{proof} 
		Since $C$ is  an ${\bf m}$-vertex cover of $G$, we have $|N_G[i]\cap C|\leq k_i-1$ for all $i\in C$ and  $|N_G[i]\cap C|= k_i-1$ for all $i\in [n]\setminus C$ which imply that 
		\begin{equation}\label{star1}|N_G[i]\cap ([n]\setminus C)|\geq \deg(i)+2-k_i,\,\, \, \, \,  \forall i\in C\end{equation} 
		and 
		\begin{equation}\label{star2}|N_G[i]\cap ([n]\setminus C)|=\deg(i)+2-k_i,\,\, \, \, \, \, \forall i\in [n]\setminus C.\end{equation} 
		Now, \eqref{star1} and \eqref{star2}
		 together with  Lemma~\ref{primaryDecom} implies that the ideal $(x_i: \ i\in [n]\setminus C)$ is a minimal prime ideal of $\NI_{\bf k}(G)$. So,
		$\textbf{x}_{[n]\setminus C}$ is a minimal monomial generator of  $\NI_{\bf k}(G)^\vee$  of degree $n-|C|$. Combining this with \cite[Theorem 2.1]{T} we obtain  $\pd((S/\NI_{\bf k}(G))=\reg(\NI_{\bf k}(G)^\vee)\geq n-|C|$. Using this and the Auslander-Buchsbaum formula, we get the inequality
		$\depth(S/\NI_{\bf k}(G))\leq |C|$.
	\end{proof}

	\section{$\bf k$-neighborhood ideal for the degree-vector}\label{sec1}
	
	 For a graph $G$ on $[n]$ and without isolated vertices, the vector ${\bf k}=(k_1,\ldots,k_n)$ with $k_i=\deg(i)$ for all $i$ is called the {\em degree-vector} of $G$.
	In this section, we study the regularity, projective dimension and Cohen-Macaulayness of $S/\NI_{\bf k}(G)$, when   ${\bf k}$ is the degree-vector of $G$.

	\smallskip
	
	In \cite[Theorem 3.2]{MS}, it was shown that   $\reg(S/\NI(G)) = \tau(G)$ for any chordal graph $G$. In Proposition \ref{chordalReg} we demonstrate that, in contrast to the case of the closed neighborhood ideal, when $\mathbf{k}$ is the degree-vector and $G$ is chordal, the difference $\operatorname{reg}(S/\NI_{\mathbf{k}}(G)) - \tau(G)$ can be arbitrarily large.

	\begin{Proposition}\label{chordalReg}
		For any positive integer $m$, there exists a connected chordal graph $G$ such that $$\reg(S/\NI_{\bf k}(G))=\tau(G)-m,$$ where  ${\bf k}$ is the degree-vector of $G$. 
	\end{Proposition}
	
	\begin{proof}
		Let $m=1$. Take the graph $G=K_{n}$ for some positive integer $n\geq 2$. Then $\tau(G)=n-1$ and $\NI_{\bf k}(G)=I_{n,n-1}$. Hence, $\reg(S/\NI_{\bf k}(G))=n-2=\tau(G)-1$. Now, assume that $m\geq 2$. 
		Let $H=K_{2m+1}$ be the complete graph on $[2m+1]$, and $G$ be the graph obtained from $H$ by adding two new vertices $a$ and $b$, and adding the edges $\{1,a\}$ and $\{i,b\}$ for any $2\leq i\leq m+2$. Then $G$ is a chordal graph. Moreover, since $m\geq 2$, we have $m+2<2m+1$. Thus $m+3$ is a vertex of $G$ with $\{a,m+3\}\notin E(G)$ and $\{b,m+3\}\notin E(G)$. It follows that $\{1,\ldots,2m+1\}\setminus \{m+3\}$ is a minimal vertex cover of $G$ of cardinality $2m$. Hence, $\tau(G)\leq 2m$. On the other hand, $G$ contains a clique of size $2m+1$. So any minimal vertex cover of $G$ has cardinality at least $2m$. Thus we obtain $\tau(G)=2m$. 
		
		Since $N_G[a]=\{1,a\}$, we have $x_1,x_a\in \NI_{\bf k}(G)$. Consider a generator $u_{j,G}/x_t$ of $\NI_{\bf k}(G)$ with $1\leq j\leq 2m+1$. Then  
		$\textbf{x}_{[2m+1]}=\prod_{\ell=1}^{2m+1}x_{\ell}\,|\, u_{j,G}$. So if $t\neq 1$, then $u_{j,G}/x_t$ is a multiple of $x_1\in \NI_{\bf k}(G)$. If $t=1$, then $\prod_{\ell=2}^{2m+1}x_{\ell}\,|\, (u_{j,G}/x_t)$, and hence
		$u_{b,G}/x_{b}=\prod_{\ell=2}^{m+2}x_{\ell}$ divides $u_{j,G}/x_t$. Therefore, we have $\NI_{\bf k}(G)=(x_1,x_a)+L,$ where $$L=(u_{b,G}/x_t:\, t\in[2,m+2]\cup\{b\}).$$ Notice that the ideal $L$ is the squarefree Veronese ideal $I_{m+2,m+1}$ in the set of variables $\{x_2,\ldots,x_{m+2}\}\cup\{x_b\}$.  We set $S'=K[x_2,\ldots,x_{m+2},x_b]$. Then $$\reg(S/\NI_{\bf k}(G))=\reg(S'/L)=\reg(L)-1=(m+1)-1=m=\tau(G)-m.$$      
	\end{proof}
	
	We denote  by $c(G)$ the number of connected components of $G$.
	In view of Proposition \ref{chordalReg}, we state the following conjecture.

	\begin{Conj}\label{chordalbound}
		Let $G$ be a chordal graph without isolated vertices, and let ${\bf k}$ be the degree-vector of $G$. Then $$\reg(S/\NI_{\bf k}(G))\leq \tau(G)-c(G).$$ 
	\end{Conj}
	
	\begin{Remark}\label{connected}
		Let $G$ be a graph  without isolated vertices and with the connected components  $G_1,\ldots,G_c$. Then $\reg(S/\NI_{\bf k}(G))=\sum_{i=1}^c \reg(S/\NI_{\bf k}(G_i))$. Hence, to prove Conjecture~\ref{chordalbound}, it suffices to consider connected chordal graphs. More precisely, it is enough to show that for every connected chordal graph $G$, and the degree-vector ${\bf k}$ of $G$, $$\reg(S/\NI_{\bf k}(G))\leq \tau(G)-1.$$ 
	\end{Remark}
	
	In the following, we verify the conjecture for some families of chordal graphs. By Remark~\ref{connected}, we may restrict our attention to connected chordal graphs.

	\smallskip 
	
	Recall that a \textit{matching} of a graph $G$ is a set of pairwise disjoint edges of $G$, and the \textit{matching number} of $G$, denoted by $a(G)$, is the maximum size of a matching of $G$. An {\em induced matching} of $G$ is a matching of $G$ which forms an induced subgraph of $G$. The {\em induced matching number} of $G$ is the maximum size of an induced matching of $G$ and is denoted by $\im(G)$.
	
	\begin{Lemma}\label{path}
		Let  $G$ be a path graph on $n\geq 2$ vertices, and let ${\bf k}$ be the degree-vector of $G$. Then  $\reg(S/\NI_{\bf k}(G))=\lfloor{\frac{n-2}{4}}\rfloor$. In particular, $\reg(S/\NI_{\bf k}(G))\leq \tau(G)-1.$
	\end{Lemma}
	
	\begin{proof} 
		Consider a path graph $G: 1,2,\ldots,n$. Then
		$\NI_{\bf k}(G)=(x_1,x_2,x_{n-1},x_n,I(H))$, where
		$$I(H)=(x_ix_{i+1}: 3\leq i\leq n-3)+(x_ix_{i+2}: 3\leq i\leq n-4)$$ is the edge ideal of a graph $H$. Since $I(H)$ lives in $K[x_3,x_4,\ldots,x_{n-2}]$, we have  $\reg(S/\NI_{\bf k}(G))=\reg(S/I(H))$. 
		Observe that $H$ is a chordal graph. So by \cite[Corollary 6.9]{HV}, we have $\reg(S/I(H))=\im(H).$ It is easy to see that the set
		$$\{\{4k-1,4k\}: \, 1\leq k\leq \lfloor{\frac{n-2}{4}}\rfloor\}$$
		is an induced matching of $H$ of maximum cardinality. Thus $\im(H)=\lfloor{\frac{n-2}{4}}\rfloor$. This implies that  $\reg(S/\NI_{\bf k}(G))=\reg(S/I(H))=\lfloor{\frac{n-2}{4}}\rfloor.$
		
		The inequality $\reg(S/\NI_{\bf k}(G))\leq \tau(G)-1$ follows from the first statement and the facts that $\tau(G)=\lfloor{\frac{n}{2}}\rfloor$ and $\lfloor{\frac{n-2}{4}}\rfloor\leq \lfloor{\frac{n}{2}}\rfloor-1$ for any $n\geq 2$.  
	\end{proof} 
	
	 A {\em caterpillar} graph is a tree that has a path $P$, such that any vertex of the graph is either on $P$, or joined to that. Such a path is called a {\em central path} of the graph.

	\begin{Theorem}\label{caterpillar}
		Let  $G$ be a caterpillar graph, and let ${\bf k}$ be the degree-vector of $G$. Then $\reg(S/\NI_{\bf k}(G))\leq \tau(G)-1.$
	\end{Theorem}
	
	\begin{proof}
		Let $P$ be a central path of $G$, let $j_1,\ldots,j_r$ be the leaves of $G$, and let $i_1,\ldots,i_t$ be the vertices on $P$ which are adjacent to some leaf of $G$.
		Set $A=\{j_1,\ldots,j_r,i_1,\ldots,i_t\}$ and $B=V(G)\setminus A$. Then
		$$\NI_{\bf k}(G)=(x_i:\, i\in A)+(u_{i,G}/x_p:\, i\in  B,\, \, x_p|u_{i,G}).$$ 
		Observe that $\deg_G(i)=2$ for any vertex $i\in B$. Moreover, the induced subgraph $G[B]$ is the disjoint union of some path graphs, say $P_1,\ldots,P_h$, with $h\leq t-1$ such that each $P_\ell$ is a subgraph of $P$ with the property that the two endpoints (leaves) of $P_\ell$ are adjacent (in the graph $G$) to two vertices in $\{i_1,\ldots,i_t\}$, say $i_{\alpha_\ell}$ and $i_{\beta_\ell}$. By assumption, there are vertices $j_{\alpha_\ell},j_{\beta_\ell}\in A$ of degree one such that $\{i_{\alpha_\ell},j_{\alpha_\ell}\}\in E(G)$ and $\{i_{\beta_\ell},j_{\beta_\ell}\}\in E(G)$.
		For any $1\leq \ell\leq h$, let $P'_\ell$ be a path in $G$ obtained from $P_\ell$ by adjoining the vertices $i_{\alpha_\ell},i_{\beta_\ell},j_{\alpha_\ell},j_{\beta_\ell}$.    
		We claim that 
		\begin{equation}\label{P}
			\NI_{\bf k}(G)=(x_i:\, i\in A)+\sum_{\ell=1}^h\NI_{\bf k}(P'_\ell).
		\end{equation}
		
		Any vertex $s\in B$ is a vertex of $P'_\ell$ for some $\ell$. 
		Moreover, $N_G(s)=N_{P'_\ell}(s)$ which implies that $u_{s,G}=u_{s,P'_\ell}$.  Thus for any $s\in B$  we have  $u_{s,G}/x_p\in \NI_{\bf k}(G)$ if and only if
		$u_{s,G}/x_p\in \NI_{\bf k}(P'_\ell)$ for some $\ell$. On the other hand, for any  $s\in V(P'_\ell)$ with $s\notin B$, we have $s\in \{i_{\alpha_\ell},j_{\alpha_\ell}\}$. Thus, $x_{i_{\alpha_\ell}}x_{j_{\alpha_\ell}}|\, u_{s,G}$ which implies that $u_{s,G}/x_p\in (x_{i_{\alpha_\ell}},x_{j_{\alpha_\ell}})\subseteq (x_i:\, i\in A)\subseteq \NI_{\bf k}(G)$. This proves (\ref{P}).
		
		From (\ref{P}), \cite[Corollary 4.8]{CHHK} and \cite[Corollary 3.2]{Her}  we get  
		$$\reg(S/\NI_{\bf k}(G))\leq   \reg(S/\sum_{\ell=1}^h\NI_{\bf k}(P'_\ell))\leq  \sum_{\ell=1}^h \reg(S/(\NI_{\bf k}(P'_\ell)).$$ Let $n_\ell=|V(P_\ell)|$ for all $\ell$. Then 
		$|V(P'_\ell)|=n_\ell+4$, and by Lemma~\ref{path}, we have
		$\reg(S/(\NI_{\bf k}(P'_\ell))=\lfloor{\frac{n_\ell+2}{4}}\rfloor$ for all $\ell$. So $\reg(S/\NI_{\bf k}(G))\leq \sum_{\ell=1}^h \lfloor{\frac{n_\ell+2}{4}}\rfloor$. On the other hand,  $\tau(P_\ell)=\lfloor{\frac{n_\ell}{2}}\rfloor$ for any $1\leq \ell\leq h$. Thus for any minimal vertex cover $C$ of $G$, we have $|C\cap V(P_\ell)|\geq \lfloor{\frac{n_\ell}{2}}\rfloor$. Since $B$ is the disjoint union of $V(P_\ell)$'s, we get 
		$|C\cap B|\geq \sum_{\ell=1}^h\lfloor{\frac{n_\ell}{2}}\rfloor$.
		Moreover, $|C\cap A|\geq t$, since for any vertex $i_q\in \{i_1,\ldots,i_t\}$ and a leaf $j_{q'}$ adjacent to it, at least one of $i_q$ and   $j_{q'}$ belongs to $C$. 
		Therefore, $|C|=|C\cap B|+|C\cap A|\geq \sum_{\ell=1}^h\lfloor{\frac{n_\ell}{2}}\rfloor+t$ for any minimal vertex cover $C$ of $G$. This implies that 	\begin{align*}
			\tau(G)-1\geq \sum_{\ell=1}^h\lfloor{\frac{n_\ell}{2}}\rfloor+t-1\ &\geq\ \sum_{\ell=1}^h\lfloor{\frac{n_\ell}{2}}\rfloor+h\\
			&=\ \sum_{\ell=1}^h\lfloor{\frac{n_\ell+2}{2}}\rfloor\geq \sum_{\ell=1}^h\lfloor{\frac{n_\ell+2}{4}}\rfloor\geq \reg(S/\NI_{\bf k}(G)). 
		\end{align*}
	\end{proof} 
	
Next, we prove Conjecture~\ref{chordalbound} for two families of chordal graphs called clique-star graphs and generalized caterpillar graphs. 

Let $G$ be a graph, and let $\{F_1,\ldots,F_r\}$ be the set of maximal cliques of $G$. We say that $F_1$ is a {\em central clique} of $G$, if
$F_i\cap F_1\neq \emptyset$ for any $2\leq i\leq r$. 
Moreover, we call $G$ a {\em clique-star graph}, if it has a central clique, say $F_1$ and $F_i\cap F_j\subseteq F_1$ for any distinct integers  $i,j\in [r]$. A split graph is an example of a clique-star graph. Recall that a {\em split graph} is a graph in which the vertices can be partitioned into a clique and an independent set. 

A clique-star graph is shown in Figure~\ref{fig:clique-star}.

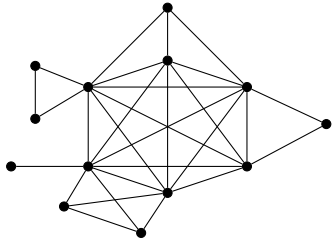
\begin{figure}[ht]
	\centering
	\begin{tikzpicture}[
		vertex/.style={circle,draw,fill=black,inner sep=1.2pt},
		every label/.style={font=\small}
		]
		
		\node[vertex] (c1) at (0,1.4) {};
		\node[vertex] (c2) at (1.05,1.75) {};
		\node[vertex] (c3) at (2.1,1.4) {};
		\node[vertex] (c4) at (2.1,0.35) {};
		\node[vertex] (c5) at (1.05,0) {};
		\node[vertex] (c6) at (0,0.35) {};
		
		\foreach \i/\j in {c1/c2,c1/c3,c1/c4,c1/c5,c1/c6,
			c2/c3,c2/c4,c2/c5,c2/c6,
			c3/c4,c3/c5,c3/c6,
			c4/c5,c4/c6,
			c5/c6}{
			\draw (\i) -- (\j);
		}
		
		\node[vertex] (a1) at (-0.7,0.98) {};
		\node[vertex] (a11) at (-0.7,1.68) {};
		\draw (a1)--(c1)--(a11)--(a1);
		
		\node[vertex] (b1) at (1.05,2.45) {};
		\draw (b1)--(c2);
		\draw (c1)--(b1)--(c3);
		
		\node[vertex] (b2) at (3.15,0.91) {};
		\draw (b2)--(c3);
		\draw (b2)--(c4);
		
		\node[vertex] (b3) at (-0.32,-0.18) {};
		\node[vertex] (d3) at (-1.02,0.35) {};
		\node[vertex] (b4) at (0.70,-0.53) {};
		
		\draw (c5)--(b3)--(c6)--(d3);
		\draw (c5)--(b4)--(c6);
		\draw (b4)--(b3);
		
	\end{tikzpicture}
	\caption{A clique-star graph with the central clique $K_6$.}
	\label{fig:clique-star}
\end{figure}

A graph $G$ is called a \emph{generalized caterpillar graph} if there exists a path graph $P$ such that $G$ is obtained by attaching complete graphs $K_1, K_2, \ldots, K_t$ to $P$, where each attachment is performed by identifying either a vertex or an edge of $K_i$ with a vertex or an edge of $P$. 

An example of a generalized caterpillar graph is illustrated in Figure~\ref{fig:gen-caterpillar}.

\begin{figure}[ht]
	\centering
	\begin{tikzpicture}[
		vertex/.style={circle,draw,fill=black,inner sep=1.4pt},
		every label/.style={font=\small}
		]
		
		\foreach \i in {1,...,7} {
			\node[vertex] (v\i) at (1.2*\i,0) {};
		}
		
		\draw (v1)--(v2)--(v3)--(v4)--(v5)--(v7);
		
		\node[vertex] (a1) at (1.5,1.05) {};
		\node[vertex] (b1) at (0.8,1.05) {};
		\draw (v1)--(a1)--(b1)--(v1);
		
		\node[vertex] (a2) at (2.4,1.05) {};
		\node[vertex] (b2) at (3.55,1.05) {};
		\node[vertex] (c2) at (2.94,1.5) {};
		\node[vertex] (d2) at (1.5,-0.95) {};
		\node[vertex] (e2) at (2.48,-0.95) {};
		\node[vertex] (f2) at (3.14,-0.95) {};
		\node[vertex] (g2) at (4.24,-0.95) {};
		\draw (v2)--(a2)--(b2)--(v3)--(f2);
		\draw (v3)--(g2)--(f2);
		\draw (a2)--(c2)--(b2)--(a2);
		\draw (d2)--(v2)--(b2);
		\draw (e2)--(v2)--(c2)--(v3)--(a2);
		\node[vertex] (a3) at (6.56,1.05) {};
		\node[vertex] (b3) at (6.58,-0.95) {};
		\draw (v5)--(a3)--(v6);
		\draw (v5)--(b3)--(v6);
		
	\end{tikzpicture}
	\caption{A generalized caterpillar graph obtained from a path graph $P_7$.}
	\label{fig:gen-caterpillar}
\end{figure}
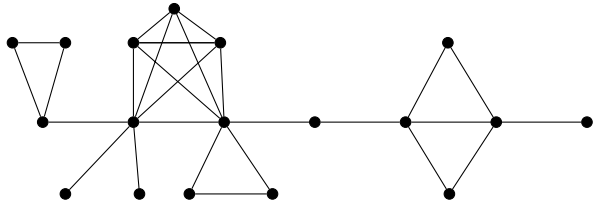

	
	\begin{Theorem}\label{clique-star}
		Let $G$ be one of the following graphs:
		\begin{enumerate}
			\item[(a)]  A clique-star graph.
			\item[(b)] A generalized caterpillar graph. 
		\end{enumerate} 
		Let ${\bf k}$ be the degree-vector of $G$.  Then $\reg(S/\NI_{\bf k}(G))\leq \tau(G)-1$. 
	\end{Theorem}

	\begin{proof} 
		Let $V(G)=[n]$. We prove the statement by induction on $n\geq 2$. If $n=2$, then $\NI_{\bf k}(G)=(x_1,x_2)$. So $\reg(S/\NI_{\bf k}(G))=0=\tau(G)-1$. Now, let $n>2$. Since $G$ is chordal, it has a simplicial vertex. 
		We consider the following cases:
		
		\textbf{Case 1}. 
		Assume that any simplicial vertex of $G$ has degree one. We discuss (a) and (b) separately in this case.
		
		(a) Let $F=\{1,2,\ldots,m\}$ be the central clique of $G$, and $F_1,\ldots,F_q$ be the maximal cliques of $G$ other than $F$. The assumption of case 1 implies that $|F_\ell|=2$ for any $1\leq \ell\leq q$. Then $F_\ell=\{i_\ell,j_\ell\}$, where  $i_\ell\in F$ for all $\ell$ (with $i_\ell$'s not necessarily distinct). Then $x_{i_1},\ldots,x_{i_q},x_{j_1},\ldots,x_{j_q}\in \NI_{\bf k}(G)$. If $m=1$, then $G$ is a star graph and $\NI_{\bf k}(G)$ is generated by variables. So 
		$\reg(S/\NI_{\bf k}(G))=0=\tau(G)-1$. Let $m\geq 2$. 
		If $i_1=\cdots=i_q$, then $$\NI_{\bf k}(G)=(x_{i_1},x_{j_1},\ldots,x_{j_q},(x_1\cdots x_m)/x_{i_1}).$$
	   Therefore, $\reg(S/\NI_{\bf k}(G))=m-2=\tau(G)-1$, and the assertion holds. Now, assume that $i_\ell\neq i_{\ell'}$ for some $\ell$ and $\ell'$. Then for any $j\in F$
		we have $x_{i_\ell}x_{i_{\ell'}}\,|\,u_{j,G}$. So  $u_{j,G}/x_k\in (x_{i_\ell},x_{i_{\ell'}})$ for any $x_k$ dividing $u_{j,G}$. Hence,
		$\NI_{\bf k}(G)=(x_{i_1},\ldots,x_{i_q},x_{j_1},\ldots,x_{j_q})$ and  $\reg(S/\NI_{\bf k}(G))=0\leq \tau(G)-1$.
		
		\smallskip
		
		(b) Let $G$ be a generalized caterpillar graph. Since any simplicial vertex of $G$ has degree one, $G$ is indeed a caterpillar graph. So by Theorem \ref{caterpillar} the assertion holds.


		\smallskip
		
		\textbf{Case 2}.  
		Assume that $G$ has a simplicial vertex $s$ with $\deg(s)\geq 2$. Let  $N_G(s)=\{i_1,\ldots,i_r\}$.  
		By \cite[Corollary 4.8]{CHHK}, we have 
		\begin{equation}\label{i_1}
			\reg(S/\NI_{\bf k}(G))\leq \max\{\reg(S/(\NI_{\bf k}(G):x_{i_1}))+1,\reg(S/(\NI_{\bf k}(G),x_{i_1}))\}.
		\end{equation} 	
		
		First we show that $\NI_{\bf k}(G)=\sum_{j\in V(G)} (\NI(G):x_j)$. Any minimal monomial generator of $\NI_{\bf k}(G)$ is of the form $u_{i,G}/x_j$, where $i\in V(G)$ and $j\in N_G[i]$. Hence, $u_{i,G}/x_j\in (\NI(G):x_j)$. On the other hand, for any monomial generator $u_{i,G}:x_j$ of $(\NI(G):x_j)$ we have either $u_{i,G}:x_j=u_{i,G}\in \NI(G)\subset \NI_{\bf k}(G)$ or $u_{i,G}:x_j=u_{i,G}/x_j\in \NI_{\bf k}(G)$.  
		
		Set $G_\ell=G\setminus x_{i_\ell}$ for any $1\leq \ell\leq r$.
		Then 
		\begin{equation}\label{hava}
			(\NI_{\bf k}(G):x_{i_\ell})=\sum_{j\in V(G)} ((\NI(G):x_j):x_{i_\ell})=\sum_{j\in V(G)} ((\NI(G):x_{i_\ell}):x_j).
		\end{equation} 	
		
		The inclusion $N_G[s]\subseteq N_G[{i_\ell}]$ implies that $u_{s,G}| u_{i_\ell,G}$ for all $1\leq \ell\leq r$. Hence,
		$$\NI(G)=(u_{s,G})+(u_{j,G}:\, j\in V(G)\setminus\{s,i_1,\ldots,i_\ell\}).$$
		We have $(u_{s,G}:x_{i_\ell})=u_{s,G_{\ell}}$. Moreover, for any vertex $j\in V(G)\setminus\{s,i_1,\ldots,i_r\}$,  if  $j\in N_G(i_\ell)$, then we have $(u_{j,G}:x_{i_\ell})=u_{j,G}/x_{i_\ell}=u_{j,G_{\ell}}$, and if $j\notin N_G(i_\ell)$, then $(u_{j,G}:x_{i_\ell})=u_{j,G}=u_{j,G_{\ell}}$. 
		Therefore, $$(\NI(G):x_{i_\ell})=(u_{s,G_{\ell}})+(u_{j,G_{\ell}}:\, j\in V(G_{\ell})\setminus\{s,i_1,\ldots,i_r\})=\NI(G_{\ell}).$$ 
		This together with (\ref{hava}) implies that 
		$(\NI_{\bf k}(G):x_{i_\ell})=\sum_{j\in V(G)} (\NI(G_{\ell}):x_j)$.
		Since $x_{i_\ell}$ divides no minimal monomial generator of $\NI(G_{\ell})$, we have $(\NI(G_{\ell}):x_{i_\ell})=\NI(G_{\ell})$. Hence, 
		$\sum_{j\in V(G)} (\NI(G_{\ell}):x_j)=\sum_{j\in V(G_{\ell})} (\NI(G_{\ell}):x_j)=\NI_{\bf k}(G_{\ell})$. So we obtain $(\NI_{\bf k}(G):x_{i_\ell})=\NI_{\bf k}(G_{\ell})$ for all $1\leq \ell\leq r$.  
		
		Observe that if $G$ is a clique-star (respectively, a generalized caterpillar graph), then $G_{\ell}$ is a graph whose connected components are clique-star (respectively, generalized caterpillar) graphs for any $1\leq \ell\leq r$. Thus
		by induction hypothesis, $$\reg(S/(\NI_{\bf k}(G):x_{i_\ell}))=\reg(S/(\NI_{\bf k}(G_{\ell})))\leq\tau(G_{\ell})-c(G_{\ell}).$$
		We show that $\tau(G_{\ell})+1\leq \tau(G)$ for any $1\leq \ell\leq r$.
		Let $D$ be a minimal vertex cover of $G$ with $\tau(G)=|D|$. If $i_\ell\in D$, then $D\setminus \{i_\ell\}$ is a vertex cover of $G_{\ell}$. Thus  $\tau(G_{\ell})\leq |D|-1=\tau(G)-1$. If $i_\ell\notin D$, then $$\{s,i_1,\ldots,i_{\ell-1},i_{\ell+1},\ldots,i_r\}\subseteq N_G(i_\ell)\subseteq D.$$ So $D'=(D\setminus\{s\})\cup \{i_\ell\}$ is a vertex cover of $G$ with $i_\ell\in D'$ and $|D'|=|D|=\tau(G)$. Thus the same argument as in the case $i_\ell\in D$ implies that $\tau(G_{\ell})+1\leq\tau(G)$. Using this fact, we get the inequality
		$$\reg(S/(\NI_{\bf k}(G):x_{i_1}))\leq\tau(G_{1})-c(G_1)\leq \tau(G_{1})-1\leq \tau(G)-2.$$  This together with  (\ref{i_1}) implies that
		
		\begin{equation}\label{eq:1}
			\reg(S/\NI_{\bf k}(G))\leq \max\{\tau(G)-1,\reg(S/(\NI_{\bf k}(G),x_{i_1}))\}.
		\end{equation} 	
		
		\smallskip
		
		To complete the proof, it remains to show that $\reg(S/(\NI_{\bf k}(G),x_{i_1}))\leq \tau(G)-1$. For any $1\leq \ell\leq r$, we set $J_\ell=(\NI_{\bf k}(G),x_{i_1},\ldots,x_{i_\ell})$. Applying \cite[Corollary 4.8]{CHHK} to $J_\ell$ we have  
		\begin{equation}\label{J_i}
			\reg(S/J_\ell)\leq \max\{\reg(S/(J_\ell:x_{i_{\ell+1}}))+1,\reg(S/J_{\ell+1})\}.
		\end{equation}

		For each $1\leq \ell\leq r-1$, we have $$(J_\ell:x_{i_{\ell+1}})=((\NI_{\bf k}(G):x_{i_{\ell+1}}),x_{i_1},\ldots,x_{i_\ell})=(\NI_{\bf k}(G_{\ell+1}),x_{i_1},\ldots,x_{i_\ell}),$$ and hence by \cite[Corollary 4.8]{CHHK}, $\reg(S/(J_\ell:x_{i_{\ell+1}}))\leq \reg(S/\NI_{\bf k}(G_{\ell+1}))\leq\tau(G_{\ell+1})-c(G_{\ell+1})$.
		Using this, the equality $\tau(G_{\ell+1})+1\leq\tau(G)$ and the inequalities (\ref{J_i}), we obtain
		\begin{equation}\label{inductIneq} 
			\reg(S/J_\ell)\leq \max\{\tau(G)-1,\reg(S/J_{\ell+1})\}
		\end{equation} 	
		for any $1\leq \ell\leq r-1$. 
		Since $r=\deg(s)\geq 2$, for any $j\in \{s,i_1,\ldots,i_r\}$ and any $x_p$ dividing $u_{j,G}$ we have $u_{j,G}/x_p\in (x_{i_1},\ldots,x_{i_r})$. Similarly, for any $j\in \{i_1,\ldots,i_r\}$ and any $x_p$ dividing $u_{j,G\setminus s}$ we have  $(u_{j,G\setminus s})/x_p\in (x_{i_1},\ldots,x_{i_r})$. 
		This implies that
		$J_r=(\NI_{\bf k}(G),x_{i_1},\ldots,x_{i_r})=(\NI_{\bf k}(G\setminus s),x_{i_1},\ldots,x_{i_r})$. If $G$ is a  clique-star (respectively, a generalized caterpillar) graph, then $G\setminus s$ is a clique-star (respectively, a generalized caterpillar) graph. So by induction hypothesis, 
		$\reg(S/(\NI_{\bf k}(G\setminus s))\leq\tau(G\setminus s)-1\leq \tau(G)-1.$ Using this and \cite[Corollary 4.8]{CHHK}, we obtain
		$$\reg(S/J_r)\leq \reg(S/(\NI_{\bf k}(G\setminus s))\leq\tau(G\setminus s)-1\leq \tau(G)-1.$$ Applying this and multiple using of inequalities (\ref{inductIneq}), we conclude that $\reg(S/J_1)\leq \tau(G)-1$. Combining the last inequality with (\ref{eq:1}), we get $\reg(S/\NI_{\bf k}(G))\leq \tau(G)-1$.  
	\end{proof}

	Next, we present a family of graphs for which the ring $S/\NI_{\bf k}(G)$ is Cohen-Macaulay, and $\reg(S/\NI_{\bf k}(G))$ can be described combinatorially. 
	A vertex $v$ is called a {\em free vertex} of $G$, if it belongs to exactly one maximal clique of $G$. We denote by $\delta(G)$ the minimum degree of a vertex in $G$. 
	
	\begin{Theorem}\label{chordalCMgeneralized}
		Let $G$ be a graph without isolated vertices whose vertex set $V(G)$ is the  disjoint union of maximal cliques of $G$  which admit a free vertex. Let $\ell$ be a positive integer with $\ell\leq \delta(G)$ and ${\bf k}=(k_1,\ldots,k_n)$ with $k_i=\deg(i)+1-\ell$. Then the following statements hold.
		\begin{enumerate}
			\item [(a)] 	$S/\NI_{\bf k}(G)$ is Cohen-Macaulay. 
			
			\item [(b)] $\reg(S/\NI_{\bf k}(G))=n-r(\ell+1)$, where $r$ is the number of maximal cliques of $G$ admitting a free vertex. 
			
			\item [(c)] All minimal $(\ell+1)$-tuple dominating sets of $G$ have the same cardinality.
		\end{enumerate}	 
		
		In particular, if $\bf k$ is the degree-vector of $G$, then $S/\NI_{\bf k}(G)$ is Cohen-Macaulay and $\reg(S/\NI_{\bf k}(G))=n-2r$.  
	\end{Theorem}
	
	\begin{proof}
		(a) By assumption, $V(G)=\bigsqcup_{i=1}^r V(G_i)$, where $G_1,\ldots, G_r$ are maximal cliques of $G$ and  each $G_i$  has a free vertex $z_i$. 
		We claim that $$\NI_{\bf k}(G)=\sum_{i=1}^r I_i,$$ where each $I_i$ is a  squarefree Veronese ideal generated by all squarefree monomials of degree 
		$k_{z_i}=\deg(z_i)+1-\ell$ in the set of variables $\{x_j: \, j\in V(G_i)\}$.
		Since for each free vertex $z\in V(G_i)$, $\deg(z)=\deg(z_i)$ and $N_G[z]=N_G[z_i]=V(G_i)$, we have $$(\textbf{x}_W: \, W\subseteq N_G[z],\,  |W|=k_z,\,  z {\text{  is a free vertex of}}\ G_i)=I_i.$$
		
		Now suppose $h$ is a non-free vertex of $G_i$. Then $$|N_G[h]\setminus V(G_i)|=\deg(h)-\deg(z_i),$$ and $$k_h=\deg(h)+1-\ell=\deg(z_i)+1-\ell+\deg(h)-\deg(z_i)=k_{z_i}+\deg(h)-\deg(z_i).$$
		So, for any $W\subseteq N_G[h]$ with $|W|=k_h$ we have $|W\cap V(G_i)|\geq k_{z_i}$. Otherwise, $$k_h=|W|=|W\cap V(G_i)|+|W\cap (V(G)\setminus V(G_i))|<k_{z_i}+|N_G[h]\setminus V(G_i)|=k_h,$$ which is a contradiction. Hence, $|W\cap V(G_i)|\geq k_{z_i}$  which ensures that $\textbf{x}_W\in I_i$. This shows that  $\NI_{\bf k}(G)=\sum_{i=1}^r I_i$.

		Let $S_i=K[x_j: \, j\in V(G_i)]$. Then $S/\NI_{\bf k}(G)=\bigotimes_{i=1}^r (S_i/I_i)$. Since each $S_i/I_i$ is Cohen-Macaulay, so is  $S/\NI_{\bf k}(G)$.
		
			\smallskip
		
		(b) As it was shown in (a), we have  $\NI_{\bf k}(G)=\sum_{i=1}^r I_i,$ where each $I_i$ is a  squarefree Veronese ideal generated by all squarefree monomials of degree 
		$k_{z_i}=\deg(z_i)+1-\ell$ in the set of variables  $\{x_j: \, j\in V(G_i)\}$. Thus $$\reg(S/\NI_{\bf k}(G))=\sum_{i=1}^r \reg(S_i/\NI_{\bf k}(G_i))=
		\sum_{i=1}^r (\deg(z_i)-\ell).$$
		
		Let $n_i=|V(G_i)|$ for all $i$. Since $z_i$ is a free vertex of the clique $G_i$, we have $\deg(z_i)=n_i-1$ for all $i$. Thus  
		$\sum_{i=1}^r (\deg(z_i)-\ell)=\sum_{i=1}^r (n_i-1-\ell)=n-r(\ell+1)$.
		\smallskip
		
		(c) follows from (a), Lemma \ref{height},   and the fact that any Cohen-Macaulay ideal is unmixed, that is all of its minimal prime ideals have the same height. 
		 
		\smallskip
		
		The last statement follows by taking $\ell=1$. 
	\end{proof}
	
	Let $G$ be an arbitrary graph, and let $\ell$ and ${\bf k}$ be as in Theorem~\ref{chordalCMgeneralized}. It may occur that $S/\NI_{\bf k}(G)$ is Cohen-Macaulay even though $V(G)$ is not a disjoint union of maximal cliques of $G$ admitting free vertices. The following example illustrates this fact.

	\begin{Example}
		Let $G$ be the graph on $[8]$ with $$E(G)=\{\{1,2\},\{2,3\},\{2,4\},\{4,5\},\{5,6\},\{6,7\},\{6,8\}\},$$ and let $\ell=1$. Then  $\NI_{\bf k}(G)=(x_1,x_2,x_3,x_6,x_7,x_8,x_4x_5)$ is a complete intersection ideal. However, $V(G)$ is not the disjoint union of maximal cliques of $G$ admitting free vertices.
	\end{Example}
	
We finish this section with the following two corollaries which present lower bounds for $\pd(S/\NI_{\bf k}(G))$ for products of path graphs using the concept of ${\bf m}$-vertex cover.
	
		\begin{Corollary}\label{product1}
	Let $G=P_2\times P_r$ for an odd number $r$, and let ${\bf k}$ be the degree-vector of $G$. Then 	$$
	\pd(S/\NI_{\bf k}(G))\geq r+1
	\qquad\text{and}\qquad
	\depth(S/\NI_{\bf k}(G))\leq r-1.
	$$
	\end{Corollary}
	
	\begin{proof}
	We may write $V(G)=\{x_1,\ldots,x_r,y_1,\ldots,y_r\}$ and
	 $$E(G)=\{\{x_i,x_{i+1}\}\  : 1\leq i<r\}\cup \{\{y_i,y_{i+1}\}\  : 1\leq i<r\}\cup \{\{x_i,y_i\}\  : 1\leq i\leq r\}.$$
		Consider the subset $$C=\{x_{2q},y_{2q}\ : \  1\leq q\leq \lfloor\frac{r}{2}\rfloor\}.$$ For each vertex $z\in V(G)$, let $m_z=\deg(z)-1$, and consider the vector ${\bf m}\in \mathbb{N}^{2r}$ defined by $m_z$'s. In other words, ${\bf m}={\bf k}-{\bf 1}$, where ${\bf 1}=(1,\ldots,1)\in \mathbb{N}^{2r}$. Then $C$ is an ${\bf m}$-vertex cover of $G$. Indeed, for any $z\in C$ we have $	|N_G[z]\cap C|=2=m_z$, and for any  $z\notin C$ we have either $|N_G[z]\cap C|=m_z=2$ or $|N_G[z]\cap C|=m_z=1$. Hence, $G$ has  an ${\bf m}$-vertex cover of cardinality $r-1$.
		 Therefore, the desired inequalities follow immediately from Proposition~\ref{m-vertex} and the Auslander-Buchsbaum formula.
	\end{proof}

		\begin{Corollary}\label{product2}
	Let  $G=P_2\boxtimes P_r$ be the strong product of the path graphs $P_2$ and $P_r$, and let ${\bf k}$ be the degree-vector of $G$. Then
	 $$\pd(S/\NI_{\bf k}(G))\geq  \begin{cases}
			r, & \text{if } r\ \text{is even}, \\
				r-1, & \text{if } r\ \text{is odd}.
			\end{cases}$$ and $$\depth(S/\NI_{\bf k}(G))\leq  \begin{cases}
				r, & \text{if } r\ \text{is even}, \\
				r+1, & \text{if } r\ \text{is odd}.
			\end{cases}$$
	\end{Corollary}
	\begin{proof}
	Let the vertex set of  $G$ be $\{x_1,\ldots,x_r,y_1,\ldots,y_r\}$ and its edge set be  $$
\bigl\{\{x_i,x_{i+1}\},\{x_i,y_{i+1}\},\{y_i,y_{i+1}\},\{y_i,x_{i+1}\}\;\big|\;1\leq i<r\bigr\}
\;\cup\;
\bigl\{\{x_i,y_i\}\;\big|\;1\leq i\leq r\bigr\}.$$ Define $C=\{x_{2q-1},y_{2q-1}\ : \  1\leq q\leq \lceil\frac{r}{2}\rceil\}$, and consider the vector ${\bf m}={\bf k}-{\bf 1}$, where ${\bf 1}=(1,\ldots,1)\in \mathbb{N}^{2r}$.
For any $z\notin C$  we have either $|N_G[z]\cap C|=4=m_z$, or 
 $|N_G[z]\cap C|=2=m_z$. Moreover, for any  $z\in C$ we have either $|N_G[z]\cap C|=2<m_z=4$ or $|N_G[z]\cap C|=2=m_z$.
 Thus $C$ is an ${\bf m}$-vertex cover of $G$ of cardinality $r$ if $r$ is even and of cardinality $r+1$ if $r$ is odd. 
 So by Proposition \ref{m-vertex} the conclusion follows.
	\end{proof}

	\section{$\bf k$-neighborhood ideal as an edge ideal}\label{sec2}
	
	In this section, we study $\bf k$-neighborhood ideals in the cases that the ideal appears as an edge ideal of some graph. First we will present some vectors $\bf k$ for which $\NI_{\bf k}(G)$ is an edge ideal.

Throughout this section $G$ is a graph  without isolated vertices on the vertex set $[n]$. We denote the vector ${\bf k}\in \mathbb{N}^n$  with $k_i=2$ for all $1\leq i\leq n$ by $\textbf{2}$.  It follows from the definition of the $\bf k$-neighborhood ideal that for  ${\bf k}={\bf 2}\in \mathbb{N}^n$, 
	\begin{equation}\label{graphCase}
		\NI_{\bf k}(G)=\sum_{i=1}^n\, (\textbf{x}_W:\, W\subseteq N_G[i]),\, |W|=2)=I(\widetilde{G}),
	\end{equation}
where $\widetilde{G}$ is a graph with the vertex set $V(G)$ and the edge set 
	$$E(\widetilde{G})=E(G)\cup\{\{i,j\}: i,j\in N_G(p) \textrm{ for some } p\in V(G)\}.$$ 
	Notice that $\widetilde{G}$ is indeed the second power of $G$. 
	
	
	
	The following two propositions give other instances, where $\NI_{\bf k}(G)$ appears as the edge ideal of a graph.	
	
		\begin{Proposition}\label{dominatingset}
		Let $G$ be a graph  without isolated vertices on $[n]$, and let $D\subseteq [n]$ be a dominating set of $G$. Let ${\bf k}\in \mathbb{N}^n$ be the vector with
		\begin{equation}\label{eq:depthP_n}
			k_i\ =\ \begin{cases}
				2,&\textup{if}\ i\in D,\\
				\hfil \deg(i)+1,&\textup{if}\ i\notin D.
			\end{cases}
		\end{equation}
		Then $\NI_{\bf k}(G)$ is the edge ideal of a graph.
	\end{Proposition}

	\begin{proof}
Consider a monomial $\textbf{x}_{N_G[i]}\in \NI_{\bf k}(G)$, where $i\in [n]\setminus D$.
Since $D$ is a dominating set of $G$, there exists $j\in D\cap N_G(i)$. This shows that $\textbf{x}_{N_G[i]}$ is a multiple of $x_ix_j$.  Moreover, $i,j\in N_G[j]$ with $j\in D$. Since $k_j=2$, the monomial $x_ix_j$ is a minimal generator of  $\NI_{\bf k}(G)$. Therefore,
		$$  \NI_{\bf k}(G)=(x_rx_s: \, \{r,s\}\subseteq N_G[j] \textrm{ for some } j\in D) $$
		is a quadratic squarefree monomial ideal.
	\end{proof}

	\begin{Proposition}\label{edgeIdeal2}
		Let $G$ be a graph without isolated vertices on $[n]$  and $A\subseteq [n]$ be a maximal  independent set of $G$. Let ${\bf k}\in \mathbb{N}^n$ be the vector with
		\begin{equation}\label{eq:depthP_n}
			k_i\ =\ \begin{cases}
				2,&\textup{if}\ i\in A,\\
				\hfil \deg(i)+1,&\textup{if}\ i\notin A.
			\end{cases}
		\end{equation}
		The following statements hold.
		\begin{enumerate}
			\item [(a)] $\NI_{\bf k}(G)$ is the edge ideal of a graph $G^*$.
			\item [(b)] $\height(\NI_{\bf k}(G))=n-|A|$. In particular, $P=(x_i:\, i\notin A)$ is a minimal prime ideal of $\NI_{\bf k}(G)$ of minimum height. 
			\item [(c)]  $\pd(S/\NI_{\bf k}(G))\geq n-|A|$.
			\item [(d)] $\reg(S/\NI_{\bf k}(G))\leq |A|$. 
		\end{enumerate}
		
	\end{Proposition}
	\begin{proof}
	(a)	Since any maximal independent set of $G$ is a dominating set of $G$, the assertion follows from Proposition \ref{dominatingset}.
		\smallskip
		
	(b)	As was shown in the proof of Proposition \ref{dominatingset}, $  \NI_{\bf k}(G)=I(G^*),$
		where $$E(G^*)=\{ \{r,s\}:\, \{r,s\}\subseteq N_G[j] \textrm{ for some }  j\in A\}.$$  
		Thus $\height(\NI_{\bf k}(G))=\tau(G^*)$.
		It follows from the structure of $G^*$ that $A$ is a maximal independent set of $G^*$, too. Thus the set $[n]\setminus A$ is a minimal vertex cover of $G^*$.
		We claim that it is indeed a vertex cover of  $G^*$ of minimum cardinality, that is $\tau(G^*)=n-|A|$. Choose a vertex cover $S$ of  $G^*$ with $|S|=\tau(G^*)$. For any $i\in A$  the induced subgraph of $G^*$ on the vertex set ${N_{G^*}[i]}$ is a complete graph. Consequently, $|S\cap N_{G^*}[i]|=\deg_{G^*}(i)$.  Thus, for any $i\in S\cap A$ there exists $j\in  N_{G^*}(i)\setminus S$. With such choice of $j$ the set $(S\setminus\{i\})\cup\{j\}$ is still  a vertex cover of $G^*$ of cardinality $\tau(G^*)$. 
		Repeating this replacement procedure for every vertex in $S\cap A$, we obtain a vertex cover of $G^*$ of cardinality $\tau(G^*)$ contained entirely in $[n]\setminus A$. Since $[n]\setminus A$ is itself a minimal vertex cover of $G^*$, it follows that $\tau(G^*)=n-|A|$. 
		\smallskip
	
	(c) follows from (b) and the fact that height of an ideal $I\subset S$ is a lower bound for the projective dimension of $S/I$. 
	\smallskip	
		
	(d)	Let $A = \{j_1,\ldots,j_\ell\}$, and let
		$G_1^*,\ldots,G_\ell^*$ be complete graphs on the vertex sets
		$N_G[j_1],\ldots,N_G[j_\ell]$, respectively.
		Then $
		E(G^*) = \bigcup_{r=1}^\ell E(G_r^*)$, and clearly each $G_r^*$ is cochordal. 
	Thus the cochord number $\operatorname{cochord}(G^*)$ of $G^*$, in the sense of \cite{W}, is at most $\ell=|A|$. So by \cite[Theorem 1]{W} we  get $$\reg(S/\NI_{\bf k}(G))=\reg(S/I(G^*))\leq \operatorname{cochord}(G^*) \leq |A|.$$
	\end{proof}

In the sequel, we concentrate on $\NI_{\bf k}(G)$ for the vector ${\bf k}=\textbf{2}\in \mathbb{N}^n$ that is $k_i=2$ for all $i$. By (\ref{graphCase}), we have   
$\NI_{\bf k}(G)=I(\widetilde{G})$.
Hence, in order to study $\NI_{\bf k}(G)$, we need to explore how the properties of $G$ and $\widetilde{G}$ are related. 

For two edges $e,e'\in E(G)$, we define the {\em distance} of $e$ and $e'$ in $G$ as the minimum integer $r$ such that the sequence $e,e_1,\ldots,e_r,e'$ forms the edges of a path in $G$ connecting $e$ and $e'$. We denote this number by $\dist_G(e,e')$. A {\em distance-three matching} of $G$ is a set $M$ of edges of $G$ such that for any two distinct edges $e,e'\in M$, we have $\dist_G(e,e')\geq 3$.  We set $a_3(G)$ to be the maximum size of a distance-three matching of $G$. 
We say that a graph $G$ is {\em claw-free}, if it has no induced subgraph which is isomorphic to the complete bipartite graph $K_{1,3}$.
We start with the following proposition which relates the invariants $\im(\widetilde{G})$ and $a_3(G)$.

	\begin{Proposition}\label{claw-free}
		For any graph $G$, we have $\im(\widetilde{G})\geq a_3(G)$.	
		In particular, if $G$ is claw-free, then $\im(\widetilde{G})=a_3(G)$.
	\end{Proposition}
	
	\begin{proof}
		In order to prove the desired inequality it is enough to show that any two edges $e$ and $e'$ of $G$ which form a distance-three matching of $G$,  form an induced matching in $\widetilde{G}$. Consider such edges $e=\{u,v\}$ and $e'=\{w,z\}$ of $G$. Since $\dist_G(e,e')\geq 3$, clearly $e\cap e'=\emptyset$. Suppose by contradiction that an endpoint of $e$, say $u$ is adjacent to an endpoint of $e'$, say $w$ in the graph $\widetilde{G}$. Then $\{u,w\}\in E(\widetilde{G})\setminus E(G)$. Hence there exists a vertex $x\in V(G)$ such that $u,w\in N_G(x)$.  Note that since $\dist_G(e,e')\geq 3$, we have $x\neq v$ and $x\neq z$. Then $e,\{u,x\},\{x,w\},e'$ is a path connecting $e$ and $e'$, which contradicts to $\dist_G(e,e')\geq 3$.

		Now, assume that $G$ is claw-free. We show that $\im(\widetilde{G})=a_3(G)$. By the first statement it is enough to show that $a_3(G)\geq \im(\widetilde{G})$. Let $\im(\widetilde{G})=r$ and $M=\{e_1,\ldots,e_r\}$ be a maximal induced matching of $\widetilde{G}$. We claim that there exists an induced matching $M'$ of $\widetilde{G}$ of cardinality $r$ such that $M'\subseteq E(G)$. If $M\subseteq E(G)$, then $M'=M$ is the desired induced matching. Otherwise, there exists an edge $e_i\in M\setminus E(G)$. Let $e_i=\{u,v\}$. Since $e_i\in E(\widetilde{G})\setminus E(G)$, there exists $x\in V(G)$ such that $u,v\in N_G(x)$. Then $e'=\{u,x\}\in E(G)$. We show that $(M\setminus \{e_i\})\cup \{e'\}$ is an induced matching of $\widetilde{G}$, as well. To this aim we need to show that for any integer $1\leq j\leq r$ with $j\neq i$, the edges $e_j$ and $e'$ form an induced matching of $\widetilde{G}$. Let $e_j=\{w,z\}$.  Since $e_i$ and $e_j$ form an induced matching of $\widetilde{G}$, we have $x\notin e_j$. Thus $e'\cap e_j=\emptyset$. By contradiction assume that $e'$ and $e_j$ do not form an induced matching in $\widetilde{G}$. Then there exists an edge $e''\in E(\widetilde{G})\setminus \{e',e_j\}$ with
		$e''\subset e'\cup e_j$. Since $\{u,w\},\{u,z\}\notin E(\widetilde{G})$, it follows that either $e''=\{x,w\}$ or $e''=\{x,z\}$. Without loss of generality assume that $e''=\{x,w\}$. If $e''\in E(G)$, then $w,u\in N_G(x)$, implying $\{u,w\}\in E(\widetilde{G})$ which is a contradiction. Hence, $e''\in E(\widetilde{G})\setminus E(G)$. This means that $x,w\in N_G(y)$ for some $y\in V(G)$. Clearly, $y\neq u$ and $y\neq v$, otherwise $\{u,w\}\in E(G)$ or $\{v,w\}\in E(G)$, which is impossible. Then the induced subgraph $H$ of $G$ on the vertex set $\{u,v,x,y\}$ is the complete bipartite graph $K_{1,3}$ with the edge set $E(H)=\{\{x,u\},\{x,v\},\{x,y\}\}$. To see this note that by our assumption $\{u,v\}\notin E(G)$. If $\{u,y\}\in E(G)$, then $u,w\in N_G(y)$ and hence $\{u,w\}\in E(\widetilde{G})$. This contradicts to the fact that $e_i$ and $e_j$ form an induced matching of $\widetilde{G}$. Hence, $\{u,y\}\notin E(G)$. By a similar argument $\{v,y\}\notin E(G)$.  Thus $H=K_{1,3}$, which is a contradiction. This proves our claim that $(M\setminus \{e_i\})\cup \{e'\}$ is an induced matching of $\widetilde{G}$. If $(M\setminus \{e_i\})\cup \{e'\}\subseteq E(G)$, then we set $M'=(M\setminus \{e_i\})\cup \{e'\}$. Otherwise, repeating the same argument as above and after finitely many steps of replacing the edges in $M\setminus E(G)$ with appropriate edges in $E(G)$ one obtains the desired induced matching $M'$ of $\widetilde{G}$ with $M'\subseteq E(G)$ and $|M'|=\im(\widetilde{G})$. 
		
		After relabelling the edges of $\widetilde{G}$, we may again assume that $M'=\{e_1,\ldots,e_r\}$. We show that $M'$ forms a distance-three matching of $G$. Consider two distinct indices $i,j\in [r]$. Since $e_i$ and $e_j$ form an induced matching in $\widetilde{G}$, they form an induced matching in $G$. So $\dist_G(e_i,e_j)\geq 2$.   If $\dist_G(e_i,e_j)=2$, then there exist edges $e',e''\in E(G)$ such that $e_i,e',e'',e_j$ is a path in $G$. Let this path be $e_i=\{u,v\},\{v,a\},\{a,b\},\{b,c\}=e_j$. Then $v,b\in N_G(a)$, and hence $\{v,b\}\in E(\widetilde{G})$. This contradicts to the assumption that $e_i$ and $e_j$ form an induced matching in $\widetilde{G}$. Therefore, $\dist_G(e_i,e_j)\geq 3$. This argument shows that the edges in $M'$ are pairwise of distance $d\geq 3$. Thus $M'$ is a distance-three matching of $G$, and hence $a_3(G)\geq |M'|=\im(\widetilde{G})$.                
	\end{proof}
	
	\begin{Remark}
		The assumption of being claw-free on $G$ can not be dropped in Proposition~\ref{claw-free}. Consider the graph $G$ with the vertex set $[6]$ and the edge set $$E(G)=\{\{1,2\},\{2,3\},\{3,4\},\{4,5\},\{4,6\}\}.$$ Then $a_3(G)=1$ and $\im(\widetilde{G})=2$, as the edges $\{1,2\},\{5,6\}$ form a maximal induced matching of $\widetilde{G}$. 
	\end{Remark}
	
In Propositions \ref{block},  \ref{small cycle} and \ref{outerplane} we present families of graphs $G$ for which  $\widetilde{G}$ is a chordal graph.   We first note that, in general, the chordality of $G$ does not guarantee the chordality of $\widetilde{G}$.

\begin{Example}\label{chordalnotEnough}
	Let $G$ be the graph depicted in Figure~\ref{polygon}. Then $G$ is chordal, whereas $\widetilde{G}$ is not, since $2,4,6,8,2$ is an induced cycle in $\widetilde{G}$ of length $4$.
	
	\begin{figure}[ht]
	\centering
	\begin{tikzpicture}[
		scale=1.1,
		vertex/.style={circle,fill=black,inner sep=1.5pt}
		]
		
		\node[vertex,label=above:$1$] (v1) at (90:1) {};
		\node[vertex,label=above right:$2$] (v2) at (45:1) {};
		\node[vertex,label=right:$3$] (v3) at (0:1) {};
		\node[vertex,label=below right:$4$] (v4) at (-45:1) {};
		\node[vertex,label=below:$5$] (v5) at (-90:1) {};
		\node[vertex,label=below left:$6$] (v6) at (-135:1) {};
		\node[vertex,label=left:$7$] (v7) at (180:1) {};
		\node[vertex,label=above left:$8$] (v8) at (135:1) {};
		
		\draw  (v1)--(v2)--(v3)--(v4)--(v5)--(v6)--(v7)--(v8)--(v1);
		
		\draw  (v1)--(v3);
		\draw  (v3)--(v5);
		\draw  (v5)--(v7);
		\draw  (v1)--(v7);
		\draw  (v3)--(v7);
	\end{tikzpicture}
	\caption{A chordal graph.}
	\label{polygon}
\end{figure}
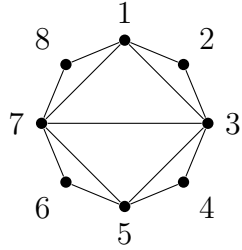
\end{Example}	
	
	\begin{Proposition}\label{block}
		If $G$ is a block graph, then $\widetilde{G}$ is a chordal graph.
	\end{Proposition}
	
	\begin{proof}
		Let $G$ be a block graph. We prove the assertion by induction on $n=|V(G)|$. If $n=1$, there is nothing to prove. Let $n>1$ and consider the clique complex $\Delta(G)$ of $G$. 
		Since $G$ is chordal, by \cite[Lemma 3.1]{HHZ}, $\Delta(G)$ has a leaf, say $F$. So there exists a maximal clique $F'$ of $G$ such that $F\cap F''\subseteq F\cap F'$ for any maximal clique $F''\neq F$ of $G$. From the assumption that $G$ is a block graph, we have $|F\cap F'|\leq 1$. 
		If $F\cap F'=\emptyset$, then $F$ is a connected component of $G$. Then we choose a vertex $x\in F$.
		The graph $H=G\setminus x$ is again a block graph and so by induction hypothesis $\widetilde{H}$ is a chordal graph. Notice that $\widetilde{G}$ is obtained from $\widetilde{H}$ by replacing the connected component $F\setminus\{x\}$ in $\widetilde{H}$ by $F$. Hence $\widetilde{G}$ is chordal too. 
		
		Now, let  $F\cap F'=\{y\}$ for some $y\in V(G)$.  	
		We choose a vertex $x\in F\setminus F'$, and set $H=G\setminus x$. By induction hypothesis $\widetilde{H}$ is a chordal graph. Note that any vertex in $F\setminus \{y\}$ is a free vertex of $G$. 
		Hence, 
		$$E(\widetilde{G})=E(\widetilde{H})\cup\{\{x,z\}: \{y,z\}\in E(G), z\neq x\}\cup\{\{x,y\}\}.$$
		
		To show that $\widetilde{G}$ is chordal, consider a cycle $C$ in $\widetilde{G}$ of length $m\geq 4$. We need to show that $C$ has a chord. If $x\notin V(C)$, then $C\subseteq V(\widetilde{H})$, and then $C$ has a chord. So suppose that $x\in V(C)$ , and write $C: u,x,v,w_1,\ldots,w_t$.
		We have $N_G(x)=F\setminus \{x\}$.  
		If $u,v\in N_G(x)$, then $\{u,v\}$ is a chord of $C$ and we are done. So we may assume that $|\{u,v\}\cap N_G(x)|\leq 1$.
		First consider the case that $|\{u,v\}\cap N_G(x)|=1$, and without loss of generality assume that 
		$\{u,v\}\cap N_G(x)=\{v\}$. Then $v\in F$ and $u\notin F$. 
		From $\{x,u\}\in E(\widetilde{G})$ and $u\notin N_G(x)$, it follows that $\{y,u\}\in E(G)$. If  $v=y$, then $\{u,v\}$ is a chord of $G$. If $v\neq y$, then $v,u\in N_G(y)$. 
		Hence, $\{u,v\}\in E(\widetilde{G})$, and it forms a chord in $C$.

		Now, assume that  $\{u,v\}\cap N_G(x)=\emptyset$. Then from $u,v\in N_{\widetilde{G}}(x)\setminus F$ we have $u,v\in N_G(y)$. Therefore, $\{u,v\}\in E(\widetilde{G})$ is a chord of $C$. 
	\end{proof}

	\begin{Proposition}\label{small cycle}
		Let  $G$  be one of the following graphs. 
		
		\begin{enumerate}
			\item [(a)] A graph with no cycle of length bigger than $5$. 
			\item [(b)] A claw-free chordal graph.
		\end{enumerate} 
		
		Then $\widetilde{G}$ is a chordal graph. 
	\end{Proposition}
	\begin{proof}
		(a) Suppose that $\widetilde{G}$ is not chordal, and let 
		$C = v_1, v_2, \ldots, v_\ell, v_1$ 
		be an induced cycle of length $\ell \ge 4$ in $\widetilde{G}$. 
		Note that if $\{v_i, v_{i+1}\} \in E(G)$, then 
		$\{v_{i-1}, v_i\}, \{v_{i+1}, v_{i+2}\} \in E(\widetilde{G}) \setminus E(G)$, 
		where the indices are taken modulo $\ell$. So, $|E(C)\cap(E(\widetilde{G})\setminus E(G))|\geq\lceil\frac{\ell}{2}\rceil\geq 2$. From the structure of $\widetilde{G}$ and by the fact that $C$ does not have any chord, for each $e_i=\{v_i,v_{i+1}\}\in E(C)\cap(E(\widetilde{G})\setminus E(G))$ one can choose  $w_i\in[n]\setminus V(C)$ such that $\{w_i,v_i\},\{w_i,v_{i+1}\}\in E(G)$ and $N_G[w_i]\cap V(C)=\{v_i,v_{i+1}\}$ (It is clear that if $i\neq j$ and $ e_i,e_j\in E(C)\cap(E(\widetilde{G})\setminus E(G))$, then $w_i\neq w_j$). One can easily check that $$(E(C)\cap E(G))\cup\{\{w_i,v_i\},\{w_i,v_{i+1}\}\, : \, \{v_i,v_{i+1}\}\in E(C)\cap(E(\widetilde{G})\setminus E(G))\}$$
		produces a cycle $C'$ of length at least $6$ in $G$ which is a contradiction.
		
		(b) Let $C = v_1, v_2, \ldots, v_\ell, v_1$ 
		be an induced cycle of length $\ell$  in $\widetilde{G}$ with the edges $e_i=\{v_i,v_{i+1}\}$ for all $i$, where $v_{\ell+1}=v_1$, and by contradiction assume that $\ell \ge 4$. Let $C'$ be the cycle in $G$ obtained from $C$, as described in (a). Since $G$ is chordal, $C'$ has a chord.
		Let $E(C)\cap(E(\widetilde{G})\setminus E(G))=\{e_{i_1},\ldots,e_{i_r}\}$. Then $r\geq 2$, $N_G[w_{i_p}]\cap V(C)=\{v_{i_p},v_{i_p+1}\}$ for all $1\leq p\leq r$, and there are indices $i_p,i_q$ such that $|i_p-i_q|\geq 2$. Indeed, if there exists an edge $e_m\in E(G)$, then $e_{m-1},e_{m+1}\in E(\widetilde{G})\setminus E(G)$, and hence $m-1,m+1\in \{i_1,\ldots,i_r\}$. 
		From the equality $N_G[w_{i_p}]\cap V(C)=\{v_{i_p},v_{i_p+1}\}$ for all $p$, we see that any chord of $C'$ in $G$ is an edge of $G$  of the form
		$\{w_{i_t},w_{i_s}\}$ for some indices $i_t$ and $i_s$. Since $G$ is claw-free, for any such chord $\{w_{i_t},w_{i_s}\}$ we have $|i_t-i_s|=1$. Let $p$ and $q$ be indices with $|i_p-i_q|\geq 2$.  Let $C''$ be a cycle of $G$ of shortest length with $V(C'')\subseteq V(C')$ and $w_{i_p},w_{i_q}\in V(C'')$. 
		Then $C''$ is an induced cycle of $G$ of length at least $4$, which is a contradiction. 
	\end{proof}

	An \emph{outerplanar graph} is a graph that can be drawn in the plane with all its vertices lying on the boundary of the outer face. When such a graph admits no further edge additions while maintaining its outerplanar property, it is called a \emph{maximal outerplanar graph}. If $G$ is a maximal outerplanar graph, then $G$ is  Hamiltonian. A maximal outerplanar graph embedded in the plane will be called a \emph{maximal outerplane graph}.
	For  a maximal outerplane graph  $G$, we denote by $ C_G$ the (unique) Hamiltonian cycle which is the boundary of the outer face.
	Let $ f$ be an inner face of a maximal outerplane graph $G$. If $f$ is adjacent to the outer face,   $f$ is called  a \emph{marginal
		triangle}; otherwise  $f$ is called  an \emph{internal triangle}. A maximal outerplane graph $G$ without internal triangles is called
	\emph{striped.} By \cite[Proposition 2]{Campos2013}, any  maximal outerplane graph $G$ of order $n\geq 4$ with
	$k$ internal triangles has $k+2$ vertices of
	degree $2$.  If $G$ is an arbitrary  maximal outerplane graph, then $\widetilde{G}$ is not necessarily a chordal graph (see Figure \ref{polygon}). In the following we show that if $G$ is a striped maximal outerplane graph, then $\widetilde{G}$ is a chordal graph.

	\begin{Proposition}\label{outerplane}
		If $G$ is  a striped maximal outerplane graph, then $\widetilde{G}$ is a chordal graph.
	\end{Proposition}
	\begin{proof}
		Let $G$ be a striped maximal outerplane graph. We prove the assertion by induction on $n=|V(G)|$.  If $n<5$, then  the result is clear. Suppose that $n \geq 5$, and let $C_G = ( v_1,v_2, \ldots, v_{n})$ be the boundary of its outer face, where the vertices of $C_G$ are labeled in clockwise order. Since $G$ is striped, it has exactly two vertices of degree $2$. Without loss of generality, assume that $\deg_G(v_1) = \deg_G(v_\ell) = 2$ $(\ell\geq 3)$. It is clear that $\{v_1,v_2,v_n\}$ is a clique of $G$.   The graph $H=G\setminus v_1$ is again a striped maximal outerplane graph and so by induction hypothesis $\widetilde{H}$ is a chordal graph.  Since $G$ does not have any  internal triangle, one of the following situations holds:
		\begin{itemize}
			\item[(i)] $\deg_G(v_n)=3$ and $\{v_2,v_{n-1}\}\in E(G)$.
			\item[(ii)] $\deg_G(v_2)=3$ and $\{v_3,v_n\}\in E(G)$.
		\end{itemize}

		Suppose we  are in the case (i) (the case (ii) can be discussed similarly). 
		Notice that $N_G[v_n]\subset N_G[v_2]$  and
		$$E(\widetilde{G})=E(\widetilde{H})\cup\{\{v_1,v_i\}: \{v_2,v_i\}\in E(G), v_i\neq v_1\}\cup\{\{v_1,v_2\}\}.$$
		
		In order to show that $\widetilde{G}$ is chordal, we consider a cycle $C$ in $\widetilde{G}$ of length $m\geq 4$ and  show that $C$ has a chord. If $v_1\notin V(C)$, then $C\subseteq V(\widetilde{H})$, and then $C$ has a chord. So suppose that $v_1\in V(C)$, and write $C: u,v_1,v,w_1,\ldots,w_t$.
		We have $N_G(v_1)=\{v_2,v_n\}$.  
		If $u,v\in N_G(v_1)$, then $\{u,v\}=\{v_2,v_n\}$ is a chord of $C$ and we are done. So we may assume that $|\{u,v\}\cap N_G(v_1)|\leq 1$.
		First consider the case that $|\{u,v\}\cap N_G(v_1)|=1$, and without loss of generality assume that 
		$\{u,v\}\cap N_G(v_1)=\{v\}$. Then $v\in \{v_2,v_n\}$ and $u\notin \{v_2,v_n\}$. 
		From $\{v_1,u\}\in E(\widetilde{G})$ and $u\notin N_G(v_1)$, it follows that $\{v_2,u\}\in E(G)$. If  $v=v_2$, then $\{u,v\}$ is a chord of $G$. If $v\neq v_2$, then  $v=v_n$. So, $v,u\in N_G(v_2)$. 
		Hence, $\{u,v\}\in E(\widetilde{G})$, and it forms a chord in $C$.

		Now, assume that  $\{u,v\}\cap N_G(v_1)=\emptyset$. Then from $u,v\in N_{\widetilde{G}}(v_1)\setminus \{v_2,v_n\}$ we have $u,v\in N_G(v_2)$. Therefore, $\{u,v\}\in E(\widetilde{G})$ is a chord of $C$. 
	\end{proof}
	
	\medskip
	
Using Proposition  \ref{claw-free}, Proposition \ref{block}, Proposition \ref{small cycle}, Proposition \ref{outerplane} and  \cite[Theorem 6.5, Theorem 6.7, Corollary 6.9]{HV} we obtain the following corollary. 
	
	\begin{Corollary}\label{regEdgeIdealCase}
		Let $G$ be a graph without isolated vertices on $[n]$, and let ${\bf k}={\bf 2}\in \mathbb{N}^n$. Then
		\begin{enumerate}
			\item [(a)] $a_3(G)\leq\reg(S/\NI_{\bf k}(G))\leq a(\widetilde{G})$.  
			
			\item [(b)] If  $G$ is a claw-free chordal graph, then $\reg(S/\NI_{\bf k}(G))=a_3(G)$. 
			
			\item [(c)] If $G$ is a block graph, then $\reg(S/\NI_{\bf k}(G))=\im(\widetilde{G})$.

			\item [(d)] If $G$ has no cycle of length bigger than $5$, then $\reg(S/\NI_{\bf k}(G))=\im(\widetilde{G})$.
			
			\item[(e)] If $G$ is a striped maximal outerplane graph, then $\reg(S/\NI_{\bf k}(G))=\im(\widetilde{G})$.
		\end{enumerate}
	\end{Corollary}
	
	\begin{proof}
		(a) By (\ref{graphCase}), we have $\reg(S/\NI_{\bf k}(G))=\reg(S/I(\widetilde{G}))$. Moreover, by \cite[Theorem 6.5, Theorem 6.7]{HV}, we have $\im(\widetilde{G})\leq\reg(S/I(\widetilde{G}))\leq a(\widetilde{G})$. Combining these  inequalities with Proposition~\ref{claw-free} yields the assertion. 
		
		(b) By Proposition \ref{small cycle}(b), $\widetilde{G}$ is a chordal graph. Hence,  using  \cite[Corollary 6.9]{HV} we have $\reg(S/I(\widetilde{G}))=\im(\widetilde{G})$.  Now, using Proposition \ref{claw-free}, we have $\im(\widetilde{G})=a_3(G)$, which implies the result.
		
		(c), (d), (e)  By Proposition \ref{block}, Proposition \ref{small cycle}(a) and Proposition \ref{outerplane}, we know that  $\widetilde{G}$ is a chordal graph. Hence, using  \cite[Corollary 6.9]{HV} we have $\reg(S/I(\widetilde{G}))=\im(\widetilde{G})$.   
	\end{proof}
	
	The following result gives a characterization for the  Cohen-Macaulayness of the ring $S/\NI_{\bf k}(T)$ for a tree $T$. A vertex $v$ in a tree $T$ is called a {\em joint vertex}, if it is adjacent to a leaf (degree-one vertex) of $T$.
	
	\begin{Theorem}\label{CMtree}
		Let  $T$ be a tree with the vertex set  $[n]$, let ${\bf k}={\bf 2}\in \mathbb{N}^n$, and let $W_T$ be the set of joint vertices of $T$.
		 Then $S/\NI_{\bf k}(T)$ is Cohen-Macaulay if and only if $|N_T[j]\cap W_T|=1$ for all $1\leq j\leq n$.
	\end{Theorem}
	\begin{proof}
We have $\NI_{\mathbf{k}}(T) = I(\widetilde{T})$, and by Proposition \ref{block}, $\widetilde{T}$ is a chordal graph. Let $F_1, \ldots, F_m$ be the facets of the clique complex $\Delta(\widetilde{T})$ that possess a free vertex. By \cite[Theorem 9.3.1]{HHBook}, the ring $S/I(\widetilde{T})$ is Cohen–Macaulay if and only if $[n]$ is the disjoint union of $F_1, \ldots, F_m$.

Observe that for any $j \in [n]$, $N_T[j]$ is a face of $\Delta(\widetilde{T})$. Moreover, $N_T[j] \subseteq N_T[i]$ if and only if $j$ is a leaf adjacent to $i$. 
We claim that $F \subseteq [n]$ is a facet of $\Delta(\widetilde{T})$ if and only if $F = N_T[i]$ for some non-leaf vertex $i$. 
To prove the claim, first we show that for a non-leaf vertex $i \in [n]$, the set $F = N_T[i]$ is a facet of $\Delta(\widetilde{T})$. By contradiction assume that this is not the case. Then there exists $j \in [n] \setminus F$ such that $F \cup \{j\}$ is a face of $\Delta(\widetilde{T})$. Hence $\{i, j\} \in E(\widetilde{T}) \setminus E(T)$, which means there exists $\ell_1 \in N_T(i) \cap N_T(j)$. Since $\deg_T(i) > 1$, we may choose a vertex $\ell_2 \neq \ell_1$ with $\ell_2 \in N_T(i)\subset F$. As $F \cup \{j\}$ is a face of  $\Delta(\widetilde{T})$ containing $j$ and $\ell_2$, we conclude that  $\{j, \ell_2\} \in E(\widetilde{T})$. Moreover, $\{j, \ell_2\} \notin E(T)$, since $T$ contains no cycle. Thus there exists $\ell_3 \in N_T(j) \cap N_T(\ell_2)$. Since $\{\ell_1, \ell_2\} \notin E(T)$, we have $\ell_3 \neq \ell_1$. Then $i, \ell_1, j, \ell_3, \ell_2, i$ forms a cycle of length $5$ in $T$, a contradiction. Hence, $F$ is indeed a facet of $\Delta(\widetilde{T})$.

Next, let $F$ be an arbitrary face of $\Delta(\widetilde{T})$ with $|F| = m \geq 3$. We prove by induction on $m$ that $F \subseteq N_T[j]$ for some non-leaf vertex $j \in [n]$.

\textbf{Base case: } Let $m = 3$. We may write $F = \{p, q, r\}$. If two of the edges among $\{p,q\}, \{p,r\}, \{q,r\}$ belong to $E(T)$, then $F \subseteq N_T[j]$ for some $j \in \{p,q,r\}$. Assume therefore that $\{p,q\}, \{p,r\} \notin E(T)$. Then there exist $\ell_1, \ell_2 \in [n] \setminus \{p,q,r\}$ such that $\{\ell_1, p\}, \{\ell_1, q\}, \{\ell_2, p\}, \{\ell_2, r\} \in E(T)$. If $\ell_1 = \ell_2$, then $F \subset N_T[\ell_1]$ and we are done. So assume $\ell_1 \neq \ell_2$. Since $T$ contains no cycle, we have $\{q, r\} \notin E(T)$. Therefore, there exists $\ell_3 \in [n] \setminus \{p,q,r\}$ such that $\{\ell_3, q\}, \{\ell_3, r\} \in E(T)$. This forces the induced subgraph of $T$ on the vertex set  $\{p,q,r,\ell_1,\ell_2,\ell_3\}$ to contain a cycle, which is a contradiction. Hence $\ell_1 = \ell_2$, and we are done.

\textbf{Inductive step: } Let $m > 3$. Choose an arbitrary vertex $p \in F$. By the induction hypothesis, there exists $j \in [n]$ such that $F \setminus \{p\} \subseteq N_T[j]$. Suppose, for contradiction, that $p \notin N_T[j]$. Choose two distinct vertices $q, r \in F \setminus \{p,j\}$. Since $\{p, q, r\}$ is a face of $\Delta(\widetilde{T})$ and $T$ contains no cycle, one of the following possibilities occurs:

(i) $\{p,q\}, \{p,r\} \in E(\widetilde{T}) \setminus E(T)$. Then there exist $\ell_1, \ell_2 \in [n] \setminus \{p,q,r,j\}$ such that $\{\ell_1, p\}, \{\ell_1, q\}, \{\ell_2, p\}, \{\ell_2, r\} \in E(T)$. It follows that the induced subgraph of $T$ on the vertex set  $\{p,q,r,\ell_1,\ell_2,j\}$ contains a cycle, a contradiction.

(ii) $\{p,q\} \in E(\widetilde{T}) \setminus E(T)$ and $\{p,r\} \in E(T)$. By an argument similar to (i), one can find a cycle in $T$, which is a contradiction.

Thus $p \in N_T[j]$ and $F \subseteq N_T[j]$. This completes the induction.

Now let $F$ be an arbitrary facet of $\Delta(\widetilde{T})$. Clearly $|F| \geq 3$. Hence, $F \subseteq N_T[j]$ for some non-leaf vertex $j$. Since $N_T[j]$ is a facet, we conclude $F = N_T[j]$. Thus the facets of $\Delta(\widetilde{T})$ are precisely the sets $N_T[j]$ corresponding to non-leaf vertices $j$.

Taking into account the above arguments,  one can easily see that an arbitrary facet $F$ of $\Delta(\widetilde{T})$ has a free vertex if and only if $F = N_T[i]$ for some $i \in W_T$. Therefore $\{F_1, \ldots, F_m\} = \{N_T[i] : i \in W_T\}$, and $|W_T| = m$. Write $W_T = \{i_1, \ldots, i_m\}$ with $F_r = N_T[i_r]$ for each $1 \leq r \leq m$. Clearly $[n]$ is the disjoint union of $F_1, \ldots, F_m$ if and only if each vertex of $G$ belongs to exactly one of the sets $F_r = N_T[i_r]$, which is equivalent to $|N_T[j] \cap W_T| = 1$ for all $1 \leq j \leq n$. This completes the proof.
	\end{proof}
	
The following corollary follows from Theorem \ref{CMtree}.
 
\begin{Corollary}
Let $T$ be a caterpillar graph with a maximal central path $1,2,\ldots,\ell$, and let ${\bf k}={\bf 2}\in \mathbb{N}^n$.
Then $S/\NI_{\bf k}(T)$ is Cohen-Macaulay if and only if $\ell\equiv 0 \pmod{3}$ and $W_T=\{3q-1\ : \ 1\leq q\leq \frac{\ell}{3}\}$.
\end{Corollary}

We conclude with the following corollaries that are analogues of Corollaries~\ref{product1} and~\ref{product2} in the setting where $\NI_{\bf k}(G)=I(\widetilde{G})$.

\begin{Corollary}\label{product3}
	Let $G=P_2\times P_r$ for an odd number $r$, and let ${\bf k}={\bf 2}\in \mathbb{N}^{2r}$. Then 	$$
	\pd(S/\NI_{\bf k}(G))\geq 2r-\lceil\frac{r}{2}\rceil
	\qquad\text{and}\qquad
	\depth(S/\NI_{\bf k}(G))\leq \lceil\frac{r}{2}\rceil.
	$$
\end{Corollary}

\begin{proof}
	We may write $V(G)=\{x_1,\ldots,x_r,y_1,\ldots,y_r\}$ and
	$$E(G)=\{\{x_i,x_{i+1}\}\  : 1\leq i<r\}\cup \{\{y_i,y_{i+1}\}\  : 1\leq i<r\}\cup \{\{x_i,y_i\}\  : 1\leq i\leq r\}.$$
	Consider the subset $$C=\{x_{4q+1}\ : \  0\leq q< \lceil\frac{r}{4}\rceil\}\cup \{y_{4q+3}\ : \  0\leq q<\lceil\frac{r-2}{4}\rceil\}.$$ For each vertex $z\in V(G)$, let $m_z=1$, and consider the vector ${\bf m}\in \mathbb{N}^{2r}$ defined by $m_z$'s. In other words, ${\bf m}={\bf 1}={\bf k}-{\bf 1}\in \mathbb{N}^{2r}$, where ${\bf 1}=(1,\ldots,1)\in \mathbb{N}^{2r}$. Then $C$ is an ${\bf m}$-vertex cover of $G$. Indeed, for any $z\in V(G)$ we have $	|N_G[z]\cap C|=1=m_z$.  Hence, $G$ has  an ${\bf m}$-vertex cover of cardinality $\lceil\frac{r}{4}\rceil+\lceil\frac{r-2}{4}\rceil=\lceil\frac{r}{2}\rceil$.
	Therefore, the desired inequalities follow immediately from Proposition~\ref{m-vertex} and the Auslander-Buchsbaum formula.
\end{proof}

\begin{Corollary}\label{product4}
	Let  $G=P_2\boxtimes P_r$ be the strong product of the path graphs $P_2$ and $P_r$, and  let ${\bf k}={\bf 2}\in \mathbb{N}^{2r}$. Then
	$$\pd(S/\NI_{\bf k}(G))\geq  2r-\lceil\frac{r}{3}\rceil 	\qquad\text{and}\qquad \depth(S/\NI_{\bf k}(G))\leq  \lceil\frac{r}{3}\rceil.$$
\end{Corollary}
\begin{proof}
	Let the vertex set of  $G$ be $\{x_1,\ldots,x_r,y_1,\ldots,y_r\}$ and its edge set be  $$
	\bigl\{\{x_i,x_{i+1}\},\{x_i,y_{i+1}\},\{y_i,y_{i+1}\},\{y_i,x_{i+1}\}\;\big|\;1\leq i<r\bigr\}
	\;\cup\;
	\bigl\{\{x_i,y_i\}\;\big|\;1\leq i\leq r\bigr\}.$$ Define $C=\{x_{3q+1}\ : \  0\leq q<\lceil\frac{r}{3}\rceil\}$ when $r\not\equiv 0\ ({\rm mod} \ 3)$, and  $C=\{x_{3q+2}\ : \  0\leq q<\lceil\frac{r}{3}\rceil\}$ when $r\equiv 0\ ({\rm mod} \ 3)$. Consider the vector ${\bf m}={\bf 1}$, where ${\bf 1}=(1,\ldots,1)\in \mathbb{N}^{2r}$.
	For any $z\in V(G)$  we have $|N_G[z]\cap C|=1=m_z$.
	Thus $C$ is an ${\bf m}$-vertex cover of $G$ of cardinality $ \lceil\frac{r}{3}\rceil$. 
	So by Proposition \ref{m-vertex} the conclusion follows.
\end{proof}

	\textbf{Acknowledgment.}  We thank Akiyoshi Sannai for carefully checking the manuscript using AI Math Referee, and for helpful remarks and corrections.

\end{document}